\documentclass[externalize]{amsart}

\usepackage[T1]{fontenc}
\usepackage[utf8]{inputenc}
\usepackage{microtype}
\usepackage[colorlinks,citecolor=blue,urlcolor=blue]{hyperref}

\usepackage[backend=bibtex,doi=true,isbn=false,url=false,
            style=alphabetic,maxbibnames=99,maxalphanames=99]{biblatex}
\usepackage{amsthm,amsmath,amsfonts,amssymb}
\usepackage{mathtools}
\usepackage{nicefrac}
\usepackage{graphicx}
\graphicspath{{figures/}}
\usepackage[config, labelfont={normalsize}]{caption,subfig}
\usepackage[backgroundcolor=yellow]{todonotes}
\usetikzlibrary{decorations.pathreplacing}
\usetikzlibrary{decorations.markings}
\usetikzlibrary{decorations.pathmorphing}
\usetikzlibrary{patterns}
\usetikzlibrary{colorbrewer}

\usepackage{pgfplots}
\pgfplotsset{compat=1.18}
\usepgfplotslibrary{colorbrewer}

\makeatletter
\@ifclasswith{amsart}{externalize}%
{
  \usetikzlibrary{external}
  \tikzexternalize[
    optimize=false,
    prefix=tikz/,
    mode=list and make]
  \makeatletter
  \renewcommand{\todo}[2][]{\tikzexternaldisable\@todo[#1]{#2}\tikzexternalenable}
  \makeatother
}%
{
  \newcommand{\tikzexternaldisable}{}
  \newcommand{\tikzexternalenable}{}
}
\makeatother

\colorlet{colA}{Paired-B}   %
\colorlet{colB}{Paired-F}   %
\colorlet{colP}{Paired-J}   %

\usepackage{subfiles}
\usepackage{aliascnt}
\usepackage{enumitem}
\usepackage[capitalize,noabbrev]{cleveref}

\numberwithin{equation}{section}

\theoremstyle{definition}
\newtheorem{definition}{Definition}[section]

\theoremstyle{plain}
\newaliascnt{theorem}{definition}
\newtheorem{theorem}[theorem]{Theorem}
\aliascntresetthe{theorem}

\newaliascnt{proposition}{definition}
\newtheorem{proposition}[proposition]{Proposition}
\aliascntresetthe{proposition}

\newaliascnt{lemma}{definition}

\aliascntresetthe{lemma}

\newaliascnt{corollary}{definition}
\newtheorem{corollary}[corollary]{Corollary}
\aliascntresetthe{corollary}

\theoremstyle{remark}
\newaliascnt{example}{definition}
\newtheorem{example}[example]{Example}
\aliascntresetthe{example}

\newaliascnt{remark}{definition}
\newtheorem{remark}[remark]{Remark}
\aliascntresetthe{remark}

\newaliascnt{conjecture}{definition}

\aliascntresetthe{conjecture}

\crefname{theorem}{Theorem}{Theorems}
\Crefname{theorem}{Theorem}{Theorems}
\crefname{lemma}{Lemma}{Lemmas}
\Crefname{lemma}{Lemma}{Lemmas}
\crefname{corollary}{Corollary}{Corollaries}
\Crefname{corollary}{Corollary}{Corollaries}
\crefname{proposition}{Proposition}{Propositions}
\Crefname{proposition}{Proposition}{Propositions}
\crefname{definition}{Definition}{Definitions}
\Crefname{definition}{Definition}{Definitions}
\crefname{example}{Example}{Examples}
\Crefname{example}{Example}{Examples}
\crefname{remark}{Remark}{Remarks}
\Crefname{remark}{Remark}{Remarks}
\crefname{conjecture}{Conjecture}{Conjectures}
\Crefname{conjecture}{Conjecture}{Conjectures}

\newcommand{\ZZ}{\mathbb{Z}}
\newcommand{\RR}{\mathbb{R}}
\newcommand{\PP}{\mathbb{P}}
\newcommand{\EE}{\mathbb{E}}

\newcommand{\ZZhalf}{\ZZ'}

\newcommand{\gs}{\varepsilon}

\newcommand{\wallseq}{\mathbf{X}}   %
\newcommand{\heirseq}{\mathbf{Y}}   %
\newcommand{\wallpt}{\mathbf{x}}    %
\newcommand{\heirpt}{\mathbf{y}}    %

\title%
[Coalescing random walks via the coalescence determinant]%
{Coalescing random walks \\ via the coalescence determinant}

\author{Piotr \'Sniady}
\address{Institute of Mathematics, Polish Academy of Sciences,
         ul.~\'Sniadeckich 8, \mbox{00-656~Warszawa,} Poland}
\email{psniady@impan.pl}

\begin{document}

\begin{abstract}
When identical particles on a line collide, they merge
and continue as one. Exact determinantal formulas have
long been available for particles conditioned never to
collide, but collisions change the number of particles,
and exact distributions for the survivors have been
obtained only in specific settings and by ad hoc
methods. Building on the coalescence determinant
introduced in a companion paper, we study the
wall-particle system: when every site is initially
occupied, this is the joint system of survivors and the
boundaries between their basins of attraction. Its
finite-dimensional distributions are determinants of
block matrices built from transition probabilities and
their cumulative sums; a finite block matrix suffices
even when the initial configuration is infinite. As
applications, we recover the Rayleigh spacing density
and the joint distribution of consecutive
gaps---which are negatively correlated---by new methods,
and give a new derivation of the determinantal formula
for the joint CDF of finitely many coalescing particles
starting from fixed positions. All formulas hold for
arbitrary nearest-neighbor random walks and their
Brownian scaling limits, with no specific transition
kernels required.
\end{abstract}



\maketitle

\section{Introduction}
\subsection{Coalescing particles}

When identical particles on a line collide, they merge
and continue as one (\Cref{fig:coalescence-intro}). This
coalescence rule appears in several areas of probability
and mathematical physics. In the voter
model~\cite{HolleyLiggett1975}, boundaries between
opinion clusters perform coalescing random walks; their
dynamics controls the approach to consensus. In
reaction-diffusion theory ($A + A \to A$), diffusing
particles that merge on contact display anomalous
kinetics: the density in one dimension decays as
$n(t) \sim t^{-1/2}$, slower than the mean-field
prediction, because spatial correlations dominate at
large times~\cite{DoeringBenAvraham1988}.

\begin{figure}[t]
\centering

\begin{tikzpicture}[scale=0.93, every node/.style={scale=0.93}]

\tikzset{
  surv/.style={black, line width=0.3pt},
  abso/.style={black, line width=0.3pt},
}

\begin{scope}
  \clip (-6.1, -0.05) rectangle (6.1, 7.05);
  \input{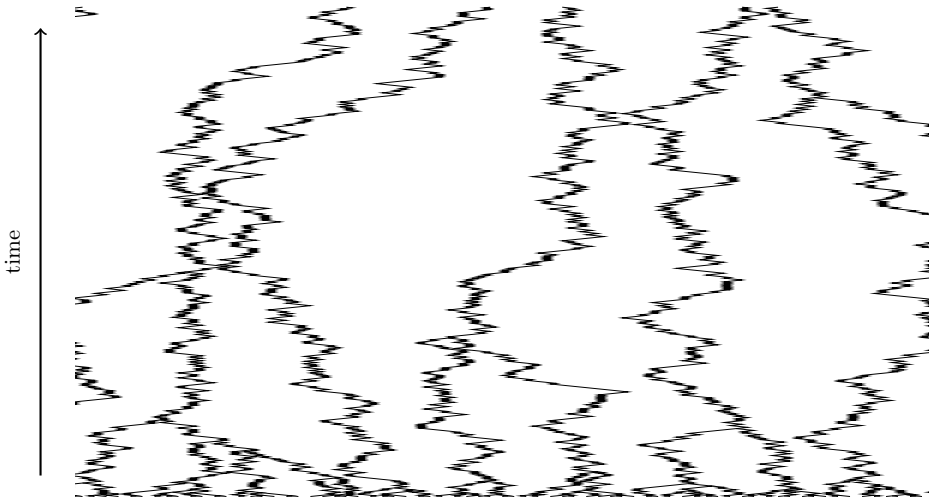}
\end{scope}

\draw[->, thick, black] (-6.6, 0.3) -- (-6.6, 6.7);
\node[rotate=90, black, font=\small] at (-7.0, 3.5) {time};

\end{tikzpicture}

\caption{Coalescing random walks starting from every site of a
lattice segment. Paths merge on contact; the surviving population
thins over time.}
\label{fig:coalescence-intro}
\end{figure}

Arratia~\cite{Arratia1979} placed the infinite
system on rigorous footing. Starting coalescing Brownian
motions from every point of~$\RR$, he showed that the
surviving population is locally finite at any positive
time: the system ``comes down from infinity.''

\subsection{The coalescence determinant}
\label{sec:intro-method}

The Karlin--McGregor theorem~\cite{KM1959} gives exact
determinantal formulas for particles that \emph{avoid}
collision. By contrast, exact formulas for coalescing
particles have been obtained only in specific settings
and by ad hoc methods. The difficulty is that collisions
change the number of particles, so the square matrices
of Karlin--McGregor do not directly apply.

\subsubsection{The formula}

The companion
paper~\cite[Section~\emph{Integrating out the ghosts}]{Sniady2026coalescence}
introduces the coalescing counterpart: the
\emph{coalescence determinant}.
Given a \emph{coalescence pattern}---which of the $n$
initial particles merge into each survivor---the
coalescence determinant gives the joint distribution of
survivor positions as the determinant of an
$n \times n$ matrix built from transition probabilities
and their cumulative sums
(\Cref{sec:coalescence-det} recalls the precise
definition).

Much of the power of the original Karlin--McGregor
theorem comes
from the weakness of its assumptions: the Markov property
and \emph{skip-free} trajectories (transitions only
to neighboring states, so that particles cannot change
order without first meeting~\cite{KM1959}). No symmetry and no
specific transition kernels are needed. The coalescence
determinant operates under exactly the same assumptions,
and therefore applies wherever Karlin--McGregor does:
we work with coalescing skip-free random walks
on~$\ZZ$ and their Brownian scaling limits on~$\RR$.

Concurrently and independently,
Urb\'an~\cite{Urban2025} proved the same formula for
binary coalescence of \emph{P\'olya walks}---a special case,
since P\'olya walks are birth-and-death chains. His
proof reaches the determinant from the opposite
direction, starting from Karlin--McGregor for
non-colliding walks and handling coalescence via a
total-probability decomposition;
see~\cite{Sniady2026coalescence} for a detailed
comparison.

\subsubsection{This paper}

This paper explores what the coalescence determinant
yields for general skip-free processes. A companion
paper~\cite{Sniady2026pfaffian} proves that the wall
process carries a natural Pfaffian structure (for any
skip-free process), derives explicit cumulants and a
central limit theorem for the wall count, and
transfers these results to surviving particles via
checkerboard duality. Under the maximal
entrance law (every site initially occupied), we study
the \emph{wall-particle system}: the joint system of
survivors and the walls between their basins of
attraction. The coalescence determinant yields the
finite-dimensional marginals of this system in closed
form (\Cref{thm:xy-correlation-intro}); gap
distributions follow by marginalizing over wall
positions. Separately,
for any finite initial
configuration, the coalescence determinant gives
Warren's determinantal CDF formula for survivor
positions (\Cref{thm:warren-general}).

\subsection{The wall-particle system}
\label{sec:wp-intro}

\subsubsection{The construction}

We work under the maximal entrance law: every site
of~$\ZZ$ is initially occupied (for Brownian motion,
every point of~$\RR$). Each surviving particle then
owns a \emph{basin}: the contiguous set of initial
positions whose particles merged into it. The basins partition the initial
configuration, and their boundaries are the
\emph{walls}. The joint system of wall positions
$\wallseq$ and survivor positions $\heirseq$---the
\emph{wall-particle system}
(\Cref{sec:xy-system}; see \Cref{fig:xy-system})---pairs
each survivor with its walls.

Previous work has studied the two marginals
separately. Arratia~\cite{Arratia1979} identified
the walls (his ``partition points''): the
points separating groups of initial positions that
merge into the same survivor. For Brownian motion,
he proved $\wallseq \stackrel{d}{=} \heirseq$ via a
time-reversal duality---but this distributional
identity requires specific symmetry of the underlying
process. Fomichov~\cite{Fomichov2016} computed the
joint distribution of one wall position and two
cluster values for the Arratia flow, using
Karlin--McGregor determinants; he noted that this
approach cannot recover the full joint distribution
for three or more clusters.
The wall-particle correlation function
(\Cref{thm:xy-correlation-intro}) generalizes
Fomichov's computation from $k = 2$ clusters to
arbitrary~$k$, and from the Arratia flow to any
skip-free process.

\begin{figure}[t]
\centering

\begin{tikzpicture}[scale=0.75]

\def\W{11.5}  %
\def\H{6.0}   %

\draw[gray, thick] (-0.3, 0) -- (\W, 0);
\draw[gray, thick] (-0.3, \H) -- (\W, \H);

\node[gray] at (-0.8, 0) {$\cdots$};
\node[gray] at (\W + 0.5, 0) {$\cdots$};
\node[gray] at (-0.8, \H) {$\cdots$};
\node[gray] at (\W + 0.5, \H) {$\cdots$};

\draw[->, thick, black] (-1.2, 0.4) -- (-1.2, \H - 0.4);
\node[rotate=90, black] at (-1.6, \H/2) {time};

\node[right, font=\small] at (\W + 0.9, 0) {$t = 0$};
\node[right, font=\small] at (\W + 0.9, \H) {$t > 0$};

\draw[black!50, line width=0.4pt]
  (0.0, 0.0) -- (0.38, 0.75) -- (0.0, 1.5) -- (0.38, 2.25);
\draw[black!50, line width=1.0pt]
  (0.38, 2.25) -- (0.75, 3.0) -- (1.12, 3.75) -- (1.5, 4.5);
\draw[black!50, line width=1.4pt]
  (1.5, 4.5) -- (1.12, 5.25) -- (1.5, 6.0);
\draw[black!30, line width=0.4pt]
  (0.75, 0.0) -- (1.12, 0.75) -- (0.75, 1.5) -- (0.38, 2.25);
\draw[black!30, line width=0.4pt]
  (1.5, 0.0) -- (1.12, 0.75);
\draw[black!30, line width=0.4pt]
  (2.25, 0.0) -- (2.62, 0.75) -- (2.25, 1.5) -- (2.62, 2.25)
  -- (2.25, 3.0) -- (1.88, 3.75) -- (1.5, 4.5);

\draw[colA, line width=0.4pt]
  (3.0, 0.0) -- (3.38, 0.75);
\draw[colA, line width=0.7pt]
  (3.38, 0.75) -- (3.75, 1.5) -- (4.12, 2.25);
\draw[colA, line width=1.4pt]
  (4.12, 2.25) -- (4.5, 3.0) -- (4.12, 3.75) -- (3.75, 4.5)
  -- (3.38, 5.25) -- (3.0, 6.0);
\draw[colA!40, line width=0.4pt]
  (3.75, 0.0) -- (3.38, 0.75);
\draw[colA!40, line width=0.4pt]
  (4.5, 0.0) -- (4.88, 0.75) -- (4.5, 1.5) -- (4.12, 2.25);
\draw[colA!40, line width=0.4pt]
  (5.25, 0.0) -- (4.88, 0.75);
\draw[colA!40, line width=0.4pt]
  (6.0, 0.0) -- (5.62, 0.75) -- (5.25, 1.5) -- (4.88, 2.25)
  -- (4.5, 3.0);

\draw[colB, line width=0.4pt]
  (6.75, 0.0) -- (7.12, 0.75) -- (7.5, 1.5) -- (7.88, 2.25);
\draw[colB, line width=1.4pt]
  (7.88, 2.25) -- (7.5, 3.0) -- (7.12, 3.75) -- (6.75, 4.5)
  -- (7.12, 5.25) -- (7.5, 6.0);
\draw[colB!40, line width=0.4pt]
  (7.5, 0.0) -- (7.88, 0.75) -- (8.25, 1.5) -- (7.88, 2.25);
\draw[colB!40, line width=0.4pt]
  (8.25, 0.0) -- (8.62, 0.75) -- (8.25, 1.5);
\draw[colB!40, line width=0.4pt]
  (9.0, 0.0) -- (8.62, 0.75);

\draw[black!50, line width=0.4pt]
  (9.75, 0.0) -- (10.12, 0.75);
\draw[black!50, line width=0.7pt]
  (10.12, 0.75) -- (10.5, 1.5) -- (10.12, 2.25)
  -- (9.75, 3.0) -- (10.12, 3.75) -- (10.5, 4.5)
  -- (10.88, 5.25) -- (11.25, 6.0);
\draw[black!30, line width=0.4pt]
  (10.5, 0.0) -- (10.12, 0.75);

\foreach \px in {0.00, 0.75, 1.50, 2.25, 3.00, 3.75, 4.50,
                 5.25, 6.00, 6.75, 7.50, 8.25, 9.00, 9.75, 10.50} {
  \fill (\px, 0) circle (1.5pt);
}

\draw (5.25, -0.16) -- (5.25, 0.24);
\node[below=6pt, font=\small] at (5.25, -0.16) {$0$};

\foreach \bx in {2.62, 6.38, 9.38} {
  \fill (\bx, 0.30) -- ++(-0.22, -0.42) -- ++(0.44, 0) -- cycle;
}
\node[below=6pt, font=\small] at (2.62, -0.16) {$\wallpt_{-\nicefrac{1}{2}}$};
\node[below=6pt, font=\small] at (6.38, -0.16) {$\wallpt_{\nicefrac{1}{2}}$};
\node[below=6pt, font=\small] at (9.38, -0.16) {$\wallpt_{\nicefrac{3}{2}}$};

\foreach \hx in {1.50, 3.00, 7.50, 11.25} {
  \fill (\hx, \H) circle (2.5pt);
}
\node[above=4pt] at (1.50, \H) {$\heirpt_{-1}$};
\node[above=4pt] at (3.00, \H) {$\heirpt_0$};
\node[above=4pt] at (7.50, \H) {$\heirpt_1$};
\node[above=4pt] at (11.25, \H) {$\heirpt_2$};

\draw[decorate, decoration={brace, amplitude=3pt, mirror}]
  (-0.2, -1.05) -- (2.52, -1.05)
  node[midway, below=4pt, font=\footnotesize]
  {basin of $\heirpt_{-1}$};
\draw[decorate, decoration={brace, amplitude=3pt, mirror}]
  (2.72, -1.05) -- (6.28, -1.05)
  node[midway, below=4pt, font=\footnotesize]
  {basin of $\heirpt_0$};
\draw[decorate, decoration={brace, amplitude=3pt, mirror}]
  (6.48, -1.05) -- (9.28, -1.05)
  node[midway, below=4pt, font=\footnotesize]
  {basin of $\heirpt_1$};
\draw[decorate, decoration={brace, amplitude=3pt, mirror}]
  (9.48, -1.05) -- (\W + 0.2, -1.05)
  node[midway, below=4pt, font=\footnotesize]
  {basin of $\heirpt_2$};

\end{tikzpicture}

\caption{The $(\wallseq, \heirseq)$ system for coalescing random
walks. Paths coalesce on meeting; line weight increases with
each merger. Bottom: walls
$\wallseq = (\ldots, \wallpt_{-\nicefrac{1}{2}},
\wallpt_{\nicefrac{1}{2}}, \wallpt_{\nicefrac{3}{2}}, \ldots)$
(triangles) partition the initial line, with
$\wallpt_{-\nicefrac{1}{2}} < 0 \leq \wallpt_{\nicefrac{1}{2}}$.
Top: survivor positions
$\heirseq = (\ldots, \heirpt_{-1}, \heirpt_0, \heirpt_1,
\heirpt_2, \ldots)$ (circles), one per basin.}
\label{fig:xy-system}
\end{figure}
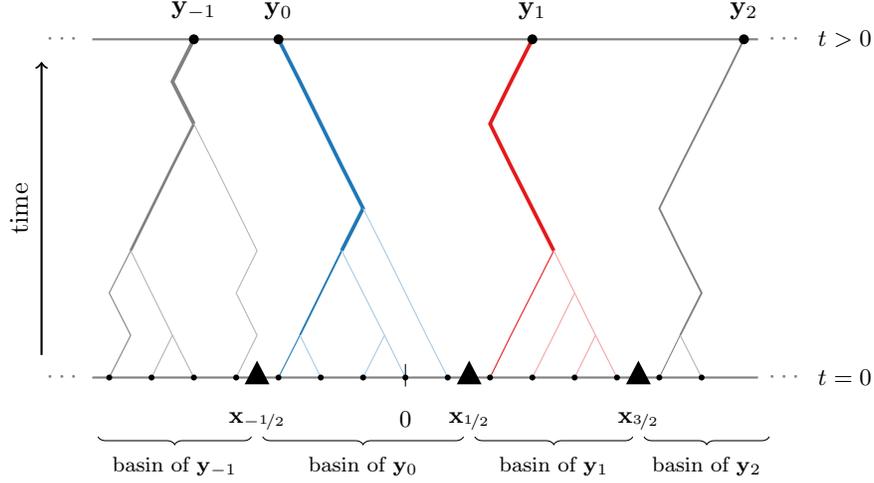

\subsubsection{Correlation function}

The coalescence determinant gives the
finite-dimensional marginals of $(\wallseq, \heirseq)$
in closed form.

\begin{theorem}[Wall-particle correlation function]
\label{thm:xy-correlation-intro}
Consider a coalescing skip-free process on~$\ZZ$ with
every site initially occupied. The probability that
$(\wallseq, \heirseq)$ contains the consecutive pattern
\[
  y_0 \nwarrow x_{\nicefrac{1}{2}} \nearrow y_1
  \nwarrow \cdots
  \nwarrow x_{k-\nicefrac{1}{2}} \nearrow y_k
\]
(walls at $x_{\nicefrac{1}{2}}, \ldots,
x_{k-\nicefrac{1}{2}}$ flanked by survivors at
$y_0, \ldots, y_k$) equals $\det(\tilde{M})$, where
$\tilde{M}$ is the $2k \times 2k$ block matrix described
in \Cref{sec:finite-marginals}.
\end{theorem}

The formula holds for any skip-free process---no
symmetry of the transition probabilities and no specific
kernels are needed. Despite the infinite initial
configuration (every site occupied), $k$~consecutive
wall-particle pairs depend only on a $2k \times 2k$
block matrix; the rest of the infinite system does not
enter the formula.

\subsection{Gap distributions}
\label{sec:intro-results}

Marginalizing the correlation function over wall
positions gives the joint distribution of consecutive
gaps between survivors. Gap distributions for coalescing
Brownian motions were first studied by Doering and
ben-Avraham~\cite{DoeringBenAvraham1988} via the
\emph{inter-particle distribution function} (IPDF)
method: the rescaled nearest-neighbor distance
density converges to $x\, e^{-cx^2}$, a Rayleigh form.
Ben-Avraham~\cite{benAvraham1998} extended the method to
derive the full hierarchy of empty-interval
probabilities; ben-Avraham and
Brunet~\cite{benAvrahamBrunet2005} extracted explicit
densities for two and three consecutive spacings.
These formulas use the explicit transition kernel
of Brownian motion.
In the asymptotic regime, FitzGerald, Tribe, and
Zaboronski~\cite{FitzGeraldTZ2020,FitzGeraldTZ2022}
computed gap exponents and persistence exponents via
Fredholm Pfaffian methods, rigorously confirming
predictions of Derrida and Zeitak.
For counting statistics, Glinyanaya and
Fomichov~\cite{GlinyanayaFomichov2018} proved a
central limit theorem for the number of surviving
clusters in the Arratia flow, with Fano factor
$3 - 2\sqrt{2} \approx 0.172$, reflecting the
sub-Poissonian correlations induced by coalescence.

\subsubsection{Discrete gap formula}

In the discrete setting, no limiting procedure is needed.

\begin{theorem}[Discrete gap intensity measure]
\label{thm:discrete-gap-intro}
Let $P_s(n)$ denote the transition probability of a
translation-invariant skip-free process on~$\ZZ$ (from~$0$
to~$n$ in time~$s$), and suppose the process is symmetric:
$P_s(n) = P_s(-n)$. Start with every site occupied. 

The gap
intensity measure---the expected number of gaps of size~$g$
per unit length---is
\[
\mu(\{g\}) = P_{2T}(g - 1) - P_{2T}(g + 1),
\qquad g = 1, 2, 3, \ldots
\]

The total intensity
\[\sum_g \mu(\{g\}) = P_{2T}(0) + P_{2T}(1)\]
is the survivor
density per site; dividing by it recovers the gap probability
mass function (\Cref{thm:discrete-gap-body}).
\end{theorem}

This formula holds for any symmetric skip-free process:
simple random walk, lazy random walk, or any birth-death
chain with symmetric rates. The only inputs are
transition probabilities at doubled time---no PDEs, no
passage to a continuous limit.
For non-symmetric processes, the gap intensity is a
convolution of transition probabilities at the original
time (\Cref{sec:discrete-single-gap}).

\subsubsection{Brownian gaps: the Rayleigh distribution}

Passing to Brownian motion, the discrete formula
becomes a continuous density.

\begin{theorem}[Single gap intensity measure]
\label{thm:rayleigh-intro}
Under the maximal entrance law, the gap intensity measure in
rescaled coordinates ($G = \text{gap}/\sqrt{T}$) has density
\[
\mu(dG) = \frac{G}{2\sqrt{\pi}}\, e^{-G^2/4}\, dG,
\qquad G > 0.
\]
The total intensity
\[\int_0^\infty \mu(dG) = \frac{1}{\sqrt{\pi}}\]
gives the rescaled survivor density; un-rescaling yields
density $n(T) = 1/\sqrt{\pi T}$. Normalizing to a probability
distribution gives $\mathrm{Rayleigh}(\sqrt{2})$.
\end{theorem}

This recovers the Rayleigh law of Doering and
ben-Avraham~\cite{DoeringBenAvraham1988} by new methods
(\Cref{sec:gaps}). In \Cref{sec:scaling-preview} we
sketch how the discrete formula from
\Cref{thm:discrete-gap-intro} converges to this density
under diffusive scaling.

\subsubsection{Joint gap distribution}

The single-gap Rayleigh law determines the marginal
distribution but says nothing about correlations between
neighboring gaps.

\begin{theorem}[Joint gap intensity]
\label{thm:joint-intro}
For two consecutive gaps $G_1$ and $G_2$, the joint gap
intensity $h(G_1, G_2)$ is given by an explicit integral formula
(\Cref{thm:joint-gap}). The marginal intensities are each
$\mathrm{Rayleigh}(\sqrt{2})$, but the gaps are negatively
correlated ($\rho \approx -0.163$).
\end{theorem}

The joint density of two consecutive spacings was
previously obtained by ben-Avraham and
Brunet~\cite{benAvrahamBrunet2005} from the IPDF
hierarchy; we give an alternative derivation via a
$4 \times 4$ determinant.

\subsection{Warren's formula}

Warren~\cite{Warren2007} proved that for finitely many
coalescing Brownian motions, the joint CDF of survivor
positions is a determinant of Gaussian CDFs and their
tails. Assiotis, O'Connell, and
Warren~\cite{AssiotisOConnellWarren2019} extended this
to general one-dimensional diffusions, and
Assiotis~\cite{Assiotis2018,Assiotis2023} to
birth-death chains. Unlike the wall-particle correlation
function---which requires the maximal entrance law
(every site occupied)---Warren's formula applies to any
finite configuration of particles at fixed starting
positions. The coalescence determinant yields the CDF
formula for any skip-free process, including
discrete-time random walks on~$\ZZ$
(\Cref{thm:warren-general}): the proof uses only the
coalescence determinant and a summation-by-parts
identity.

\subsection{Scope}

The coalescence determinant and the analytic
approaches cover complementary territory. Our
formulas apply to any skip-free process with
arbitrary inhomogeneous transition probabilities,
but the wall-particle results require the maximal
entrance law and treat only pure coalescence.
The spin-pair duality
of~\cite{GarrodPTZ2018} handles mixed
coalescence-annihilation and all deterministic
initial conditions (extended by Tribe and
Zaboronski~\cite{TribeZaboronski2026} to all
entrance laws), but requires a time-homogeneous
Markov generator.
For coalescing Brownian motions, FitzGerald, Tribe,
and
Zaboronski~\cite{FitzGeraldTZ2020,FitzGeraldTZ2022}
derive sharp gap exponents via Fredholm Pfaffian
methods, and Glinyanaya and
Fomichov~\cite{GlinyanayaFomichov2018} prove a CLT
with Berry--Esseen bound---results our approach does
not yield.

\subsection{Organization}

\Cref{sec:coalescence-det} recalls the coalescence
determinant for arbitrary patterns.
\Cref{sec:xy-system} develops the wall-particle
system and derives its correlation function for skip-free
processes on~$\ZZ$. Specializing to Brownian
motion (\Cref{sec:bm-setting}), each pair of flanking
sites collapses to a density--derivative pair as the
lattice spacing vanishes, and the maximal entrance law
provides translation invariance. Gap distributions
(\Cref{sec:gap-distributions}) follow by marginalizing
the correlation function over wall positions for
$k = 1$ and $k = 2$. Warren's formula
(\Cref{sec:warren}) is a separate application of the
coalescence determinant: summing over all coalescence
patterns converts density columns to CDF columns.

\section{The Coalescence Determinant}
\label{sec:coalescence-det}

We recall the coalescence determinant
from~\cite{Sniady2026coalescence}, which applies to any
finite collection of coalescing skip-free particles.
We state the formula in the discrete setting; the
continuous extension is recalled in
\Cref{sec:bm-setting}.

Write $P(x, y)$ for the transition probability of the
underlying skip-free process from~$x$ to~$y$ in time~$T$
(with~$T$ fixed throughout), and
\[
F(x, y) = \sum_{z \leq y} P(x, z)
\]
for the cumulative sum.

Start $n$ particles at positions
$x_1 < x_2 < \cdots < x_n$. A \emph{coalescence pattern}
is an integer composition $n_1 + n_2 + \cdots + n_k = n$:
the first~$n_1$ initial particles merge into
survivor~$1$, the next~$n_2$ into survivor~$2$, and so
on. The $l$th block of the composition---the initial
particles merging into survivor~$l$---has indices
$n_1{+}\cdots{+}n_{l-1}{+}1$ through
$n_1{+}\cdots{+}n_l$. Write $y_1, \ldots, y_k$ for the
survivor positions at time~$T$.

\begin{definition}[Coalescence matrix]
\label{def:coalescence-matrix}
Both rows and columns of the $n \times n$
\emph{coalescence matrix} $\tilde{M}$ are indexed by
$\{1, \ldots, n\}$. The entry in row~$i$, column~$j$
(where $j$ lies in the $l$th block, with survivor
position~$y_l$) is
\[
\tilde{M}_{ij} = \begin{cases}
P(x_i,\, y_l)
  & \text{if $j$ is the first index in its
  block}, \\
F(x_i,\, y_l) - [i < j]
  & \text{otherwise}.
\end{cases}
\]
The first column of each block contains transition
probabilities~$P$; the remaining $n_l - 1$ columns
contain cumulative sums~$F$ with a staircase shift.
\end{definition}

\begin{theorem}[Coalescence determinant
  {\cite{Sniady2026coalescence}}]
\label{thm:coalescence-det}
Under the coalescence pattern $n_1 + \cdots + n_k = n$,
the joint probability of survivor positions at
$y_1 < \cdots < y_k$ is $\det(\tilde{M})$.
\end{theorem}

The formula above is stated for discrete state spaces.
For continuous processes satisfying the
Karlin--McGregor assumptions (such as Brownian motion),
transition probabilities~$P$ become densities and
cumulative sums~$F$ become CDFs; the determinant then
gives a probability density rather than a probability
mass.
See~\cite{Sniady2026coalescence} for the general statement
covering both cases.

\begin{example}[Pattern $2 + 1$]\label{ex:pattern-2+1}
Three particles; the first two merge
(survivor~$y_1$), the third survives alone ($y_2$):
\[
\tilde{M} = \begin{pmatrix}
P(x_1, y_1) & F(x_1, y_1) - 1 & P(x_1, y_2) \\
P(x_2, y_1) & F(x_2, y_1) & P(x_2, y_2) \\
P(x_3, y_1) & F(x_3, y_1) & P(x_3, y_2)
\end{pmatrix}.
\]
\end{example}

\section{The Wall-Particle System}
\label{sec:xy-system}
\label{sec:pairs}%
\subsection{Two coupled sequences}

Consider a coalescing skip-free process on~$\ZZ$ with every site
initially occupied. When two particles meet, they coalesce and
continue as one. Fix a time $T > 0$. We call each surviving particle
a \emph{survivor} and
write $\heirseq = (\heirpt_j)_{j \in \ZZ}$ for the increasing
sequence of survivor positions (time-$T$ coordinates), indexed
by the integers.

Each survivor owns a \emph{basin}: the contiguous set of initial
positions whose particles merged into it. The basins
partition~$\ZZ$. Between consecutive basins lies a \emph{wall}: the
half-integer separating the last initial position in one basin from
the first in the next. Write
$\wallseq = (\wallpt_i)_{i \in \ZZhalf}$ for the increasing
sequence of walls (time-$0$ coordinates), where
$\ZZhalf = \ZZ + \tfrac{1}{2}$. The basin of $\heirpt_j$ is
the set of initial integers in the interval
$(\wallpt_{j-\nicefrac{1}{2}}, \wallpt_{j+\nicefrac{1}{2}})$.
Throughout, $\wallseq$ and $\heirseq$ denote the random
sequences, with components $\wallpt_i$ and~$\heirpt_j$.
When $x_i$ and $y_j$ appear without boldface in formulas,
they denote specific (deterministic) positions.
See \Cref{fig:xy-system} for an illustration.

Because walls live at time~$0$ and survivors at time~$T$,
there is no interlacing constraint on their values---a
survivor need not lie inside its basin. The two sequences have interleaved
indices---half-integers for walls, integers for survivors:
\[
\cdots \nwarrow \wallpt_{\nicefrac{1}{2}} \nearrow \heirpt_1
  \nwarrow \wallpt_{\nicefrac{3}{2}} \nearrow \heirpt_2
  \nwarrow \wallpt_{\nicefrac{5}{2}} \nearrow \cdots
\]
We use this \emph{zigzag notation} throughout: each arrow
connects a wall to an adjacent survivor in the index order.

\begin{remark}[Labeling convention]
\label{rem:labeling}
The labeling of walls and survivors is determined only up to a
global shift of the index set. When needed, we fix a reference
by requiring
$\wallpt_{-\nicefrac{1}{2}} < 0 \leq \wallpt_{\nicefrac{1}{2}}$,
placing the origin in the basin of $\heirpt_0$. This convention
plays no role in the distributional results.
\end{remark}

Each wall~$x_i$ ($i \in \ZZhalf$) is flanked by two initial
integer positions:
\[
a_i = x_i - \tfrac{1}{2}, \qquad\qquad b_i = x_i + \tfrac{1}{2}.
\]
Position~$a_i$ (to the left of the wall)
belongs to the basin of survivor~$y_{i-\nicefrac{1}{2}}$,
and $b_i$ (to the right) to the basin of
survivor~$y_{i+\nicefrac{1}{2}}$.

\subsection{Correlation function}
\label{sec:finite-marginals}

\begin{theorem}[Wall-particle correlation function]
\label{thm:xy-correlation}
Consider a coalescing skip-free process on~$\ZZ$ with every site
initially occupied. For positions
$x_{\nicefrac{1}{2}} < \cdots < x_{k-\nicefrac{1}{2}}$
and $y_0 < \cdots < y_k$,
the probability that $(\wallseq, \heirseq)$ contains the
consecutive pattern
\[
y_0 \nwarrow x_{\nicefrac{1}{2}} \nearrow y_1
  \nwarrow x_{\nicefrac{3}{2}} \nearrow \cdots
  \nwarrow x_{k-\nicefrac{1}{2}} \nearrow y_k
\]
equals $\det(\tilde{M})$, where $\tilde{M}$ is the
coalescence matrix (\Cref{def:coalescence-matrix}) for
$2k$ particles started at
$a_1, b_1, \ldots, a_k, b_k$ with coalescence pattern
$1{+}2{+}\cdots{+}2{+}1$: particle~$a_1$ survives alone
at~$y_0$; each pair $(b_l, a_{l+1})$ merges into
survivor~$y_l$; and $b_k$ survives alone at~$y_k$.
\end{theorem}

\begin{proof}
Consider first only the $2k$ flanking particles
$a_1, b_1, \ldots, a_k, b_k$ (no other sites occupied;
see \Cref{fig:proof-schema} for $k = 2$).
The coalescence determinant
(\Cref{thm:coalescence-det}) gives the probability of
the stated coalescence event as~$\det(\tilde{M})$.

Now populate every remaining integer site. In Arratia's
construction of coalescing
processes~\cite{Arratia1979}, additional particles
do not alter the trajectories of existing ones: they
follow the same underlying paths and merge into whatever
they meet. Because the process is skip-free, every
intermediate particle (at a site between~$b_l$
and~$a_{l+1}$) is trapped between the converging paths
of~$b_l$ and~$a_{l+1}$ and must coalesce into the same
survivor~$y_l$ (see \Cref{fig:proof-schema}). The
wall-particle event in the full system therefore has the
same probability as the coalescence event for the $2k$
flanking particles alone.
\end{proof}

\begin{figure}[t]
\centering

\begin{tikzpicture}[scale=0.66]

\def\W{12}
\def\H{5}

\tikzset{
  styleA/.style={colA, line width=2.2pt, solid},
  styleB/.style={colB, line width=1.5pt, decorate,
    decoration={zigzag, segment length=4pt, amplitude=1pt}},
  styleHeir/.style={colP, line width=0.8pt, double,
    double distance=2pt},
}

\draw[gray, thick] (-0.3, 0) -- (\W, 0);
\draw[gray, thick] (-0.3, \H) -- (\W, \H);
\node[right, font=\small] at (\W + 0.1, 0) {$t = 0$};
\node[right, font=\small] at (\W + 0.1, \H) {$t > 0$};

\draw[->, thick, gray] (-0.7, 0.3) -- (-0.7, \H - 0.3);
\node[rotate=90, gray] at (-1.0, \H/2) {time};

\coordinate (Y0) at (1.5, \H);    %
\coordinate (Y1) at (6.0, \H);    %
\coordinate (Y2) at (10.2, \H);   %

\coordinate (coll) at (5.5, 2.8);

\draw[black!30, line width=0.5pt]
  (0.0, 0) -- (0.5, 0.7) -- (0.8, 1.0);

\draw[black!30, line width=0.5pt]
  (1.0, 0) -- (0.8, 0.7) -- (0.8, 1.0);
\draw[black!40, line width=0.7pt]
  (0.8, 1.0) -- (1.2, 1.8) -- (1.5, 2.2);

\draw[styleA]
  (2.0, 0) -- (1.8, 0.8) -- (1.5, 2.2);
\draw[styleA]
  (1.5, 2.2) -- (1.3, 3.2) -- (1.5, \H);

\draw[styleA]
  (3.0, 0) -- (3.5, 0.9) -- (4.0, 1.5) -- (4.8, 2.2) -- (coll);

\draw[black!35, line width=0.5pt]
  (4.0, 0) -- (3.8, 0.6) -- (3.5, 0.9);

\draw[black!35, line width=0.5pt]
  (5.0, 0) -- (4.6, 0.8) -- (4.2, 1.3) -- (4.0, 1.5);

\draw[black!35, line width=0.5pt]
  (6.0, 0) -- (6.3, 0.7) -- (6.8, 1.2) -- (7.0, 1.4);

\draw[black!35, line width=0.5pt]
  (7.0, 0) -- (7.2, 0.5) -- (7.5, 0.8);

\draw[styleB]
  (8.0, 0) -- (7.5, 0.8) -- (7.0, 1.4) -- (6.5, 2.0) -- (coll);

\draw[styleHeir]
    (coll) -- (5.8, 3.8) -- (Y1);

\filldraw[black] (coll) circle (2.5pt);

\draw[styleB]
  (9.0, 0) -- (9.3, 1.0) -- (9.5, 1.8) -- (9.8, 2.8)
  -- (10.0, 3.8) -- (10.2, \H);

\draw[black!35, line width=0.5pt]
  (10.0, 0) -- (9.8, 0.5) -- (9.5, 0.8) -- (9.3, 1.0);

\draw[black!35, line width=0.5pt]
  (11.0, 0) -- (10.5, 0.8) -- (10.0, 1.5) -- (9.5, 1.8);

\foreach \px in {0, 1, 4, 5, 6, 7, 10, 11} {
  \fill[black!50] (\px, 0) circle (1.5pt);
}

\node[circle, fill=colA, inner sep=2.2pt] at (2.0, 0) {};
\node[circle, fill=colA, inner sep=2.2pt] at (3.0, 0) {};

\node[circle, draw=colB, fill=white, line width=1.2pt,
  inner sep=1.8pt] at (8.0, 0) {};
\node[circle, draw=colB, fill=white, line width=1.2pt,
  inner sep=1.8pt] at (9.0, 0) {};

\fill (2.5, 0.25) -- +(-0.18, -0.35) -- +(0.18, -0.35) -- cycle;
\fill (8.5, 0.25) -- +(-0.18, -0.35) -- +(0.18, -0.35) -- cycle;

\node[below=3pt, font=\small] at (2.0, 0) {$a_1$};
\node[below=3pt, font=\small] at (3.0, 0) {$b_1$};
\node[below=3pt, font=\small] at (8.0, 0) {$a_2$};
\node[below=3pt, font=\small] at (9.0, 0) {$b_2$};

\node[above=1pt, font=\small] at (2.5, 0.25)
  {$x_{\nicefrac{1}{2}}$};
\node[above=1pt, font=\small] at (8.5, 0.25)
  {$x_{\nicefrac{3}{2}}$};

\node[circle, fill=colA, inner sep=2.5pt] at (Y0) {};
\node[circle, fill=colP, inner sep=2.5pt] at (Y1) {};
\node[circle, draw=colB, fill=white, line width=1.2pt,
  inner sep=2pt] at (Y2) {};

\node[above=3pt] at (Y0) {$y_0$};
\node[above=3pt] at (Y1) {$y_1$};
\node[above=3pt] at (Y2) {$y_2$};

\begin{scope}[shift={(\W + 0.8, \H/2 - 0.5)}]
  \draw[styleA] (0, 1.8) -- (0.7, 1.8);
  \node[right, font=\footnotesize] at (0.8, 1.8)
    {pair 1 (solid)};
  \draw[styleB] (0, 1.1) -- (0.7, 1.1);
  \node[right, font=\footnotesize] at (0.8, 1.1)
    {pair 2 (zigzag)};
  \draw[styleHeir] (0, 0.4) -- (0.7, 0.4);
  \node[right, font=\footnotesize] at (0.8, 0.4)
    {survivor (double)};
  \draw[black!35, line width=0.5pt] (0, -0.3) -- (0.7, -0.3);
  \node[right, font=\footnotesize] at (0.8, -0.3)
    {intermediate};
\end{scope}

\end{tikzpicture}

\caption{Proof of \Cref{thm:xy-correlation} for $k = 2$. The
coalescence determinant applies to the four flanking particles
$a_1, b_1, a_2, b_2$ (bold paths: solid for pair~$1$, zigzag
for pair~$2$). Particles $b_1$ and~$a_2$ coalesce into
survivor~$y_1$ (double line); particles $a_1$ and~$b_2$
survive as~$y_0$ and~$y_2$. The intermediate particles (thin
gray) cannot cross the flanking paths---the skip-free property
traps them in the closing funnel between~$b_1$ and~$a_2$---so
they are absorbed into the same survivors. Adding them does not
change the coalescence outcome for the flanking particles.}
\label{fig:proof-schema}
\end{figure}
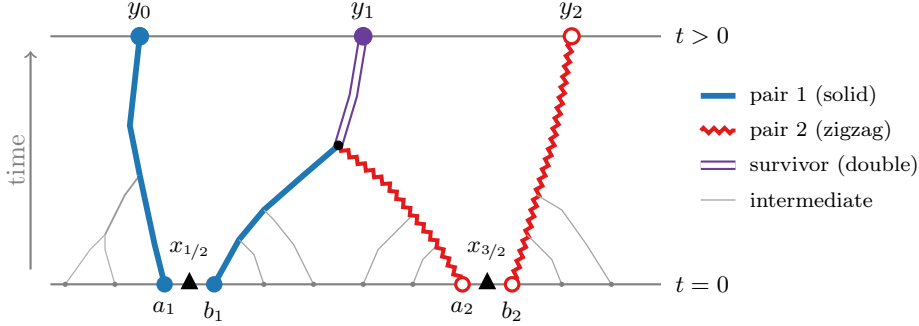

\subsection{Block structure}
\label{sec:block-matrix}%

We now examine the matrix $\tilde{M}$ from
\Cref{thm:xy-correlation} more closely.
The pattern $1{+}2{+}\cdots{+}2{+}1$ gives $\tilde{M}$
a $2 \times 2$ block structure: each wall contributes a
row-pair and each interior survivor a column-pair ($P$
and~$F$), while the two boundary survivors contribute
single $P$~columns. Write $B_{i,j}$ ($i \in \ZZhalf$, $j \in \ZZ$) for the
$2 \times 2$ block at row-pair~$i$ (wall) and
column-pair~$j$ (survivor):
\[
B_{i,j} = \begin{pmatrix}
P(a_i, y_j) & F(a_i, y_j) - [i < j] \\[0.3em]
P(b_i, y_j) & F(b_i, y_j) - [i < j]
\end{pmatrix},
\]
where $[i < j]$ is the Iverson bracket.

\begin{example}[Pattern $1{+}2{+}1$]
\label{ex:schema-1+2+1}
For $k = 2$ (two walls, three survivors), the pattern
$y_0 \nwarrow x_{\nicefrac{1}{2}} \nearrow y_1 \nwarrow
x_{\nicefrac{3}{2}} \nearrow y_2$ gives a
$4 \times 4$ matrix with column structure
$1{+}2{+}1$. The interior survivor~$y_1$ contributes a
block~$B_{i,j}$---a $P$~column and an $F$~column
carrying a staircase step---while the boundary
survivors~$y_0$ and~$y_2$ each contribute a single
$P$~column:
\[
\tilde{M} = \begin{pmatrix}
P(a_1, y_0) & P(a_1, y_1) & F(a_1, y_1) - 1
  & P(a_1, y_2) \\
P(b_1, y_0) & P(b_1, y_1) & F(b_1, y_1) - 1
  & P(b_1, y_2) \\
P(a_2, y_0) & P(a_2, y_1) & F(a_2, y_1)
  & P(a_2, y_2) \\
P(b_2, y_0) & P(b_2, y_1) & F(b_2, y_1)
  & P(b_2, y_2)
\end{pmatrix}.
\]
For the boundary case $k = 1$ (one wall, no interior
survivors, no coalescence at all), there are no blocks~$B_{i,j}$ and the
matrix reduces to the $2 \times 2$ Karlin--McGregor
determinant~\cite{KM1959}.
\end{example}

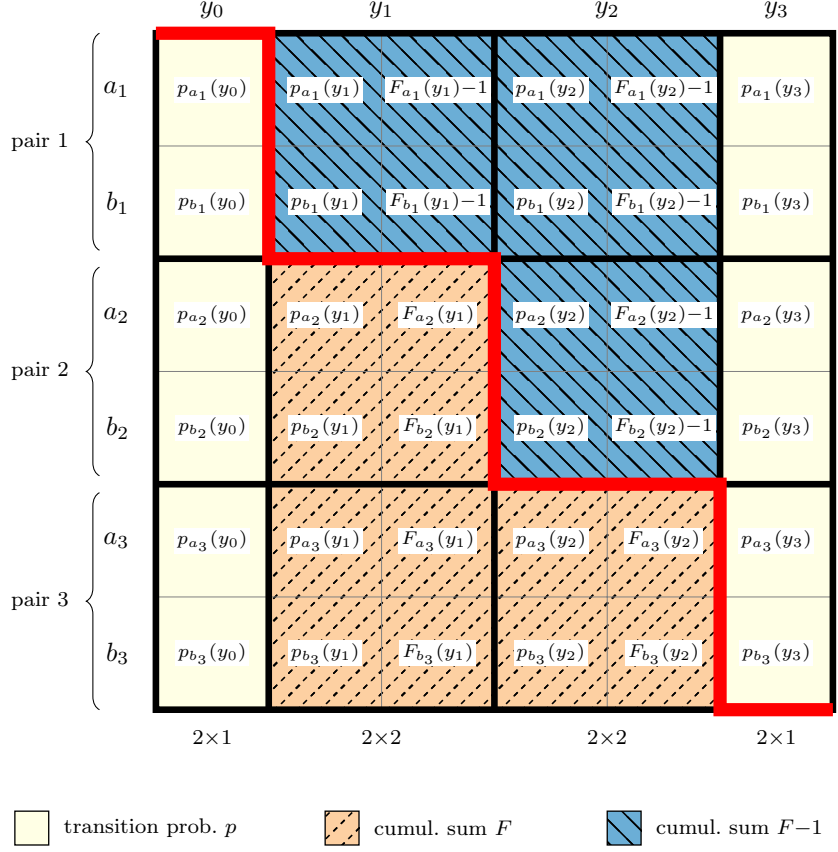
\begin{figure}[t]
\centering

\begin{tikzpicture}[scale=0.75]

\def\cellsize{2.0}

\colorlet{cbYellow}{YlGn-A}    %
\colorlet{cbOrange}{Oranges-D} %
\colorlet{cbPurple}{Blues-G}   %

\pgfdeclarepatternformonly{dashed ne lines}
  {\pgfpoint{-0.1cm}{-0.1cm}}{\pgfpoint{0.4cm}{0.4cm}}
  {\pgfpoint{0.3cm}{0.3cm}}
  {\pgfsetlinewidth{0.6pt}\pgfsetdash{{2pt}{2pt}}{0pt}
   \pgfpathmoveto{\pgfpoint{-0.1cm}{-0.1cm}}
   \pgfpathlineto{\pgfpoint{0.4cm}{0.4cm}}\pgfusepath{stroke}}

\pgfdeclarepatternformonly{solid nw lines}
  {\pgfpoint{-0.1cm}{-0.1cm}}{\pgfpoint{0.4cm}{0.4cm}}
  {\pgfpoint{0.3cm}{0.3cm}}
  {\pgfsetlinewidth{0.6pt}\pgfpathmoveto{\pgfpoint{-0.1cm}{0.4cm}}
   \pgfpathlineto{\pgfpoint{0.4cm}{-0.1cm}}\pgfusepath{stroke}}

\tikzset{
    blockHeir/.style={fill=cbYellow},
    blockBelow/.style={fill=cbOrange, postaction={pattern=dashed ne lines, pattern color=black}},
    blockAbove/.style={fill=cbPurple, postaction={pattern=solid nw lines, pattern color=black}},
    entrylabel/.style={font=\scriptsize, fill=white, inner sep=1pt},
}

\fill[blockHeir] (0, 4*\cellsize) rectangle (1*\cellsize, 6*\cellsize);
\fill[blockHeir] (0, 2*\cellsize) rectangle (1*\cellsize, 4*\cellsize);
\fill[blockHeir] (0, 0) rectangle (1*\cellsize, 2*\cellsize);

\fill[blockHeir] (5*\cellsize, 4*\cellsize) rectangle (6*\cellsize, 6*\cellsize);
\fill[blockHeir] (5*\cellsize, 2*\cellsize) rectangle (6*\cellsize, 4*\cellsize);
\fill[blockHeir] (5*\cellsize, 0) rectangle (6*\cellsize, 2*\cellsize);

\fill[blockAbove] (1*\cellsize, 4*\cellsize) rectangle (3*\cellsize, 6*\cellsize);
\fill[blockBelow] (1*\cellsize, 2*\cellsize) rectangle (3*\cellsize, 4*\cellsize);
\fill[blockBelow] (1*\cellsize, 0) rectangle (3*\cellsize, 2*\cellsize);

\fill[blockAbove] (3*\cellsize, 4*\cellsize) rectangle (5*\cellsize, 6*\cellsize);
\fill[blockAbove] (3*\cellsize, 2*\cellsize) rectangle (5*\cellsize, 4*\cellsize);
\fill[blockBelow] (3*\cellsize, 0) rectangle (5*\cellsize, 2*\cellsize);

\draw[gray, line width=0.3pt] (0, 0) grid[step=\cellsize] (6*\cellsize, 6*\cellsize);

\draw[line width=2.5pt, black] (0, 0) rectangle (6*\cellsize, 6*\cellsize);

\draw[line width=2.5pt, black] (0, 4*\cellsize) -- (6*\cellsize, 4*\cellsize);
\draw[line width=2.5pt, black] (0, 2*\cellsize) -- (6*\cellsize, 2*\cellsize);

\draw[line width=2.5pt, black] (1*\cellsize, 0) -- (1*\cellsize, 6*\cellsize);
\draw[line width=2.5pt, black] (3*\cellsize, 0) -- (3*\cellsize, 6*\cellsize);
\draw[line width=2.5pt, black] (5*\cellsize, 0) -- (5*\cellsize, 6*\cellsize);

\draw[line width=5pt, red]
    (0, 6*\cellsize) -- (1*\cellsize, 6*\cellsize) -- (1*\cellsize, 4*\cellsize) --
    (3*\cellsize, 4*\cellsize) -- (3*\cellsize, 2*\cellsize) --
    (5*\cellsize, 2*\cellsize) -- (5*\cellsize, 0) -- (6*\cellsize, 0);

\node[entrylabel] at (0.5*\cellsize, 5.5*\cellsize) {$p_{a_1}(y_0)$};
\node[entrylabel] at (1.5*\cellsize, 5.5*\cellsize) {$p_{a_1}(y_1)$};
\node[entrylabel] at (2.5*\cellsize, 5.5*\cellsize) {$F_{a_1}(y_1){-}1$};
\node[entrylabel] at (3.5*\cellsize, 5.5*\cellsize) {$p_{a_1}(y_2)$};
\node[entrylabel] at (4.5*\cellsize, 5.5*\cellsize) {$F_{a_1}(y_2){-}1$};
\node[entrylabel] at (5.5*\cellsize, 5.5*\cellsize) {$p_{a_1}(y_3)$};

\node[entrylabel] at (0.5*\cellsize, 4.5*\cellsize) {$p_{b_1}(y_0)$};
\node[entrylabel] at (1.5*\cellsize, 4.5*\cellsize) {$p_{b_1}(y_1)$};
\node[entrylabel] at (2.5*\cellsize, 4.5*\cellsize) {$F_{b_1}(y_1){-}1$};
\node[entrylabel] at (3.5*\cellsize, 4.5*\cellsize) {$p_{b_1}(y_2)$};
\node[entrylabel] at (4.5*\cellsize, 4.5*\cellsize) {$F_{b_1}(y_2){-}1$};
\node[entrylabel] at (5.5*\cellsize, 4.5*\cellsize) {$p_{b_1}(y_3)$};

\node[entrylabel] at (0.5*\cellsize, 3.5*\cellsize) {$p_{a_2}(y_0)$};
\node[entrylabel] at (1.5*\cellsize, 3.5*\cellsize) {$p_{a_2}(y_1)$};
\node[entrylabel] at (2.5*\cellsize, 3.5*\cellsize) {$F_{a_2}(y_1)$};
\node[entrylabel] at (3.5*\cellsize, 3.5*\cellsize) {$p_{a_2}(y_2)$};
\node[entrylabel] at (4.5*\cellsize, 3.5*\cellsize) {$F_{a_2}(y_2){-}1$};
\node[entrylabel] at (5.5*\cellsize, 3.5*\cellsize) {$p_{a_2}(y_3)$};

\node[entrylabel] at (0.5*\cellsize, 2.5*\cellsize) {$p_{b_2}(y_0)$};
\node[entrylabel] at (1.5*\cellsize, 2.5*\cellsize) {$p_{b_2}(y_1)$};
\node[entrylabel] at (2.5*\cellsize, 2.5*\cellsize) {$F_{b_2}(y_1)$};
\node[entrylabel] at (3.5*\cellsize, 2.5*\cellsize) {$p_{b_2}(y_2)$};
\node[entrylabel] at (4.5*\cellsize, 2.5*\cellsize) {$F_{b_2}(y_2){-}1$};
\node[entrylabel] at (5.5*\cellsize, 2.5*\cellsize) {$p_{b_2}(y_3)$};

\node[entrylabel] at (0.5*\cellsize, 1.5*\cellsize) {$p_{a_3}(y_0)$};
\node[entrylabel] at (1.5*\cellsize, 1.5*\cellsize) {$p_{a_3}(y_1)$};
\node[entrylabel] at (2.5*\cellsize, 1.5*\cellsize) {$F_{a_3}(y_1)$};
\node[entrylabel] at (3.5*\cellsize, 1.5*\cellsize) {$p_{a_3}(y_2)$};
\node[entrylabel] at (4.5*\cellsize, 1.5*\cellsize) {$F_{a_3}(y_2)$};
\node[entrylabel] at (5.5*\cellsize, 1.5*\cellsize) {$p_{a_3}(y_3)$};

\node[entrylabel] at (0.5*\cellsize, 0.5*\cellsize) {$p_{b_3}(y_0)$};
\node[entrylabel] at (1.5*\cellsize, 0.5*\cellsize) {$p_{b_3}(y_1)$};
\node[entrylabel] at (2.5*\cellsize, 0.5*\cellsize) {$F_{b_3}(y_1)$};
\node[entrylabel] at (3.5*\cellsize, 0.5*\cellsize) {$p_{b_3}(y_2)$};
\node[entrylabel] at (4.5*\cellsize, 0.5*\cellsize) {$F_{b_3}(y_2)$};
\node[entrylabel] at (5.5*\cellsize, 0.5*\cellsize) {$p_{b_3}(y_3)$};

\node[left] at (-0.3, 5.5*\cellsize) {$a_1$};
\node[left] at (-0.3, 4.5*\cellsize) {$b_1$};
\node[left] at (-0.3, 3.5*\cellsize) {$a_2$};
\node[left] at (-0.3, 2.5*\cellsize) {$b_2$};
\node[left] at (-0.3, 1.5*\cellsize) {$a_3$};
\node[left] at (-0.3, 0.5*\cellsize) {$b_3$};

\draw[decorate, decoration={brace, amplitude=5pt, mirror}]
    (-1.0, 6*\cellsize) -- (-1.0, 4*\cellsize + 0.15)
    node[midway, left=8pt, font=\footnotesize] {pair 1};
\draw[decorate, decoration={brace, amplitude=5pt, mirror}]
    (-1.0, 4*\cellsize - 0.15) -- (-1.0, 2*\cellsize + 0.15)
    node[midway, left=8pt, font=\footnotesize] {pair 2};
\draw[decorate, decoration={brace, amplitude=5pt, mirror}]
    (-1.0, 2*\cellsize - 0.15) -- (-1.0, 0)
    node[midway, left=8pt, font=\footnotesize] {pair 3};

\node[above] at (0.5*\cellsize, 6*\cellsize + 0.1) {$y_0$};
\node[above] at (2*\cellsize, 6*\cellsize + 0.1) {$y_1$};
\node[above] at (4*\cellsize, 6*\cellsize + 0.1) {$y_2$};
\node[above] at (5.5*\cellsize, 6*\cellsize + 0.1) {$y_3$};

\node[below, font=\footnotesize] at (0.5*\cellsize, -0.2) {$2{\times}1$};
\node[below, font=\footnotesize] at (2*\cellsize, -0.2) {$2{\times}2$};
\node[below, font=\footnotesize] at (4*\cellsize, -0.2) {$2{\times}2$};
\node[below, font=\footnotesize] at (5.5*\cellsize, -0.2) {$2{\times}1$};

\begin{scope}[shift={(-2.5, -2.4)}]
  \fill[blockHeir] (0, 0) rectangle (0.6, 0.6);
  \draw[black, line width=0.5pt] (0, 0) rectangle (0.6, 0.6);
  \node[right, font=\footnotesize] at (0.7, 0.3)
    {transition prob.~$p$};

  \fill[blockBelow] (5.5, 0) rectangle (6.1, 0.6);
  \draw[black, line width=0.5pt] (5.5, 0) rectangle (6.1, 0.6);
  \node[right, font=\footnotesize] at (6.2, 0.3)
    {cumul.\ sum~$F$};

  \fill[blockAbove] (10.5, 0) rectangle (11.1, 0.6);
  \draw[black, line width=0.5pt] (10.5, 0) rectangle (11.1, 0.6);
  \node[right, font=\footnotesize] at (11.2, 0.3)
    {cumul.\ sum~$F{-}1$};
\end{scope}

\end{tikzpicture}

\caption{Block structure of~$\tilde{M}$ for the
$1{+}2{+}2{+}1$ pattern ($k = 3$ walls). Rows come in
pairs; columns are grouped by survivor: $2 \times 1$
boundary blocks for~$y_0$ and~$y_3$, and $2 \times 2$
interior blocks for~$y_1$ and~$y_2$, each containing
a~$P$ column and an~$F$ column. The thick red staircase
separates~$F{-}1$ blocks (dark blue, solid hatching)
from~$F$ blocks (orange, dashed hatching);
unhatched yellow blocks contain only~$P$ entries.}
\label{fig:matrix-blocks}
\end{figure}

\begin{corollary}[Block structure of $\tilde{M}$]
\label{cor:block-structure}
The matrix $\tilde{M}$ from
\Cref{thm:xy-correlation} for a pattern with
$k$~walls has a column structure
$1{+}2{+}\cdots{+}2{+}1$: the first and last columns
are single $P$~columns (one for each boundary
survivor), and each interior survivor contributes a
$2 \times 2$ block~$B_{i,j}$.
See \Cref{fig:matrix-blocks} for the $k = 3$ case,
where the block structure and the staircase pattern
are fully apparent.
\end{corollary}

\begin{remark}[Staircase and block structure]
\label{rem:staircase-blocks}
The block formula for~$B_{i,j}$ uses the Iverson
bracket~$[i < j]$, which steps at block boundaries.
The coalescence matrix
(\Cref{def:coalescence-matrix}), however, defines
the staircase at the level of individual rows. For
the wall-particle patterns
$1{+}2{+}\cdots{+}2{+}1$, each staircase step falls
between the two rows of a single block, so the two
conventions agree on all $F$~columns (where the
distinction matters) and may differ only on
$P$~columns---but $P$~entries do not depend on the
staircase. Warren's formula (\Cref{sec:warren}),
however, sums over all $2^{n-1}$ coalescence patterns
and requires the original row-level staircase.
\end{remark}

\begin{remark}[Asymmetry between walls and survivors]
	\label{rem:xy-asymmetry}
	The block matrix~$\tilde{M}$ treats walls and survivors
	asymmetrically: the pattern has $k$ walls but $k + 1$
	survivors, and the two boundary survivors each contribute
	only one column (rather than two), so walls and survivors
	cannot simply be interchanged.
	When the underlying process has a checkerboard
	structure---such as discrete-time $\pm 1$ random
	walk---this asymmetry can be resolved: a decomposition of
	the lattice gives a duality between walls and survivors,
	connecting the wall-particle system to Pfaffian point
	processes. This is developed
	in~\cite{Sniady2026pfaffian}.
\end{remark}

\subsection{Examples}
\label{sec:examples}

We describe two classes of processes satisfying the
skip-free assumption of
\Cref{thm:xy-correlation}.

\begin{example}[Simple symmetric random walk]
\label{ex:srw-parity}
\label{rem:parity}%
Consider the $\pm 1$ simple random walk: at each
discrete time step, every particle moves left or right
with equal probability. Space-time splits into two
checkerboard sublattices: a particle at position~$x$ at
time~$t$ satisfies either $x + t \equiv 0$ or
$x + t \equiv 1 \pmod{2}$, and each particle stays on
one sublattice forever. If we start particles at
\emph{every} site of~$\ZZ$, the Karlin--McGregor
assumptions fail: a particle starting at an even site
and a particle starting at an odd site live on
complementary sublattices and can exchange order without
ever sharing a site, so the non-crossing property does
not hold.

The remedy is to occupy a single sublattice---say all
even sites $2\ZZ$ at time~$0$. Particles on the same
sublattice cannot cross without meeting (they share the
same parity at every time step), so the skip-free
assumption holds. Applying the framework to the
initial sublattice~$2\ZZ$ (with spacing~$2$ playing the role
of the unit lattice), walls sit at the odd
integers~$2\ZZ + 1$. At time~$T$, all survivors
share the same parity, so all gaps between consecutive
survivors are even
(see \Cref{thm:discrete-gap-body} for the gap
distribution).
\end{example}

\begin{example}[Birth-death chains]
\label{ex:birth-death}
Continuous-time birth-death chains on the
non-negative integers are skip-free by construction
(only $\pm 1$ transitions), and every integer is a
valid state at every time---no parity constraint.
The wall-particle framework applies directly, with
walls at $\{\tfrac{1}{2}, \tfrac{3}{2}, \ldots\}$.
For instance, the M/M/1 queue has transition
probabilities expressible via modified Bessel
functions~\cite{KM1959}.
\end{example}

\subsection{Multi-pattern correlations}
\label{sec:multi-pattern}

\Cref{thm:xy-correlation} extends to several separated
consecutive patterns observed simultaneously, with an
unspecified number of intermediate survivors between
them. We state this for completeness; it is not used
in the present paper.

\begin{theorem}[Multi-pattern correlation function]
\label{thm:multi-pattern}
Consider $m$ separated consecutive patterns in the
wall-particle system. Pattern~$\alpha$ ($\alpha = 1,
\ldots, m$) consists of $k_\alpha$ walls and
$k_\alpha + 1$ survivors:
\[
y^{(\alpha)}_0 \nwarrow x^{(\alpha)}_{\nicefrac{1}{2}}
  \nearrow y^{(\alpha)}_1
  \nwarrow \cdots
  \nwarrow x^{(\alpha)}_{k_\alpha - \nicefrac{1}{2}}
  \nearrow y^{(\alpha)}_{k_\alpha},
\]
with walls and survivors each globally increasing.
Between consecutive patterns, the number of intermediate
survivors is unspecified.

The probability that $(\wallseq, \heirseq)$ contains
all $m$ patterns simultaneously equals
$\det(\tilde{M})$, where $\tilde{M}$ is the
coalescence matrix (\Cref{def:coalescence-matrix}) for
the $2K$ flanking particles
($K = \sum_{\alpha} k_\alpha$). Within each pattern,
the coalescence is as in \Cref{thm:xy-correlation};
between consecutive patterns, the boundary particles
survive alone.
\end{theorem}

\begin{proof}
Apply the coalescence determinant
(\Cref{thm:coalescence-det}) to the $2K$
flanking positions. Within each pattern, the argument is
identical to \Cref{thm:xy-correlation}. Between
patterns, intermediate particles are trapped between the
last flanking particle of one pattern and the first of
the next, so populating them does not change the
probability.
\end{proof}

\begin{example}[Two separated patterns]
\label{ex:schema-multi-pattern}
Consider $m = 2$ patterns: pattern~$1$ with $k_1 = 2$
(walls~$x_{\nicefrac{1}{2}},
x_{\nicefrac{3}{2}}$, flanking positions
$a_1, b_1, a_2, b_2$, survivors $y_0, y_1, y_2$) and
pattern~$2$ with $k_2 = 1$
(wall~$x_{\nicefrac{5}{2}}$, flanking positions
$a_3, b_3$, survivors $y_3, y_4$), with an unspecified
number of intermediate survivors between $y_2$
and~$y_3$. The $6 \times 6$ matrix is:
\[
\tilde{M} = \left(\!
\setlength{\arraycolsep}{4pt}%
\begin{array}{@{}c||c|cc|c||c|c@{}}
  & y_0 & \multicolumn{2}{c|}{y_1} & y_2
  & y_3 & y_4 \\[0.5em]
\hline\hline
a_1 & P & P & F{-}1 & P & P & P \\[0.3em]
b_1 & P & P & F{-}1 & P & P & P \\[0.3em]
\hline
a_2 & P & P & F & P & P & P \\[0.3em]
b_2 & P & P & F & P & P & P \\[0.3em]
\hline\hline
a_3 & P & P & F & P & P & P \\[0.3em]
b_3 & P & P & F & P & P & P
\end{array}
\!\right).
\]
Here $P$ denotes a transition probability entry
$P(a_l, y_j)$ or $P(b_l, y_j)$, and $F$, $F{-}1$
denote cumulative sum entries $F(\cdot, y_1)$ or
$F(\cdot, y_1) - 1$.

The double lines mark the pattern boundary. Within
pattern~$1$ (upper-left $4 \times 4$ block), the structure
is the familiar $1{+}2{+}1$ from \Cref{ex:schema-1+2+1}:
the interior survivor~$y_1$ has an $F$~column, with the
staircase separating $F{-}1$ (wall~$1$, above) from $F$
(walls~$2$ and~$3$, below). Pattern~$2$ (lower-right
$2 \times 2$ block) is pure Karlin--McGregor.

The off-diagonal blocks couple the two patterns; their
entries are all~$P$. No $F$~column appears between $y_2$
and $y_3$: this absence is what allows an arbitrary number
of unspecified intermediate survivors between the two
patterns. Compare \Cref{fig:matrix-blocks}, where the
single pattern $1{+}2{+}2{+}1$ has $F$~columns for
\emph{every} interior survivor.
\end{example}

\section{Brownian Motion Setting}
\label{sec:bm-setting}

We specialize the block matrix~$\tilde{M}$
(\Cref{cor:block-structure}) to Brownian motion. In
the continuous limit, each flanking pair
$(a_l, b_l)$ collapses to a single wall position;
subtracting the two rows of each block and dividing
by the grid spacing turns the row pair into a
Gaussian density and its spatial derivative. The resulting $2k \times 2k$
matrix~$M_0$ is explicit. We then pass to the
maximal entrance law---coalescing Brownian motions
starting from all of~$\RR$, a classical construction
due to Arratia~\cite{Arratia1979}---and obtain a
determinantal formula for the intensity of the
wall-particle system.

\subsection{Transition densities}
\label{sec:bm-specialization}
\label{sec:scaling-limit}%

Write $p_x(y)$ for the Gaussian transition density at
time~$T$ (fixed throughout):
\[
p_x(y) = \frac{1}{\sqrt{2\pi T}}
\exp\!\left(-\frac{(y - x)^2}{2T}\right),
\]
and $F_x(y)$ for the Gaussian CDF:
\[
F_x(y) = \int_{-\infty}^{y} p_x(z)\, dz
= \Phi\!\left(\frac{y - x}{\sqrt{T}}\right),
\]
where $\Phi$ is the standard normal CDF.

Start coalescing Brownian motions from a grid with
spacing~$\gs$. Each wall is flanked by starting
positions $a_l$ and $b_l = a_l + \gs$. The coalescence
determinant (\Cref{thm:coalescence-det}) gives a
$2k \times 2k$ matrix as in
\Cref{sec:finite-marginals}, with Brownian densities
$p_x(y)$ and CDFs $F_x(y)$ in place of $P(x,y)$
and~$F(x,y)$.

Subtracting the $a_l$ row from the $b_l$ row within
each pair (which preserves the determinant) and dividing
by~$\gs$ replaces the second row by the
derivative $\partial_x$, up to $O(\gs)$ error. Each
$2 \times 2$ interior block becomes:
\[
\begin{pmatrix}
p_{x}(y) & F_{x}(y) - [\cdot] \\[0.3em]
\partial_x p_{x}(y)
& \partial_x F_{x}(y)
\end{pmatrix}
+ O(\gs),
\]
where $[\cdot]$ stands for the Iverson bracket as
in~$B_{i,j}$ (\Cref{sec:block-matrix}). The boundary
columns (containing only~$p$ entries) undergo the same
row operation, becoming
$(p_x(y),\; \partial_x p_x(y))^T$.
Each row-pair contributes one factor of~$\gs$.

\begin{proposition}[Grid refinement]
\label{prop:scaling-limit}
For coalescing Brownian motions starting from a grid with
spacing~$\gs$, the wall-particle correlation function
for $k$ walls near
$x_{\nicefrac{1}{2}}, \ldots, x_{k-\nicefrac{1}{2}}$
and $k + 1$ survivors near $y_0, \ldots, y_k$ equals
$\gs^k \det(M_0) + O(\gs^{k+1})$, where $M_0$ is
the $2k \times 2k$ matrix with alternating density and
derivative rows and the column structure
of~$\tilde{M}$. Each wall contributes one factor
of~$\gs$.
\end{proposition}

\begin{proof}
The row operations above give $\det(\tilde{M})
= \gs^k \det(M_0) + O(\gs^{k+1})$.
\end{proof}

\begin{remark}[Wronskian structure]
\label{rem:wronskian}
For translation-invariant kernels $p_x(y) = p(y - x)$,
\[
\partial_x p_x(y) = -\partial_y p_x(y),
\qquad
\partial_x F_x(y) = -p_x(y).
\]
Thus the derivative row of each block in~$M_0$ is
$-\partial_y$ applied to the density row, giving a
generalized Wronskian in the CDF~$F_x$ evaluated at~$y$.
This structure matches the Tribe--Zaboronski kernel for
coalescing Brownian
motions~\cite{TZ2011,GarrodPTZ2018}; the connection is
explored in~\cite{Sniady2026pfaffian}.
\end{remark}

\subsection{Maximal entrance law}\label{sec:maximal}
\label{sec:limiting-object}%

Arratia~\cite{Arratia1979} constructs coalescing Brownian
motions starting from all of~$\RR$ via dyadic approximation
and proves that the set of survivors is locally finite. We
use the same construction to build the joint
$(\wallseq, \heirseq)$ system.

Fix a time $T > 0$. At step $n = 0, 1, 2, \ldots$,
start independent Brownian motions from every point
of $2^{-n}\ZZ$ at time~$0$, and run them until time~$T$
with the coalescing rule: when two trajectories meet,
they merge and continue as one.

The construction is incremental. Going from step~$n$
to step~$n+1$, the new starting points
$2^{-n-1}(2\ZZ + 1)$ interleave the existing ones from
$2^{-n}\ZZ$. Each newly launched Brownian motion either
hits an existing trajectory before time~$T$ and is
absorbed, or survives to time~$T$ without hitting any
existing trajectory. The set of survivors grows
monotonically with~$n$: existing survivors are never
removed. The limiting set is locally finite: two
Brownian motions distance~$\gs$ apart fail to
coalesce by time~$T$ with probability
$\gs / \sqrt{\pi T} + O(\gs^2)$
(this is Arratia's estimate~\cite{Arratia1979};
it also follows from \Cref{prop:scaling-limit}
with $k = 1$), so the expected number of survivors
per unit length remains bounded as the grid refines.

The construction produces the $(\wallseq, \heirseq)$
system from \Cref{sec:xy-system}: the wall sequence
$\wallseq = (\wallpt_i)$ and the survivor sequence
$\heirseq = (\heirpt_j)$ (positions at time~$T$),
with basins partitioning~$\RR$. The maximal
entrance law inherits translation invariance from the
dyadic grid: the shifts by $2^{-n}$ become dense as
$n \to \infty$, so the joint law of
$(\wallseq, \heirseq)$ is invariant under all
translations of~$\RR$.

\subsection{Wall-particle intensity}
\label{sec:bm-density}

A consecutive pattern with $k$~walls and $k + 1$
survivors lives in $\RR^{2k+1}$ ($k$ wall coordinates
plus $k + 1$ survivor coordinates), and the
wall-particle system under the maximal entrance law
forms a point process on this space. Its
\emph{intensity} is the density of the point
process with respect to Lebesgue measure: integrating
over a region gives the expected number of patterns
in that region. Combining the grid-refinement limit
(\Cref{prop:scaling-limit}) with the dyadic
construction gives this intensity in closed form.

\begin{theorem}[Wall-particle intensity]
\label{thm:bm-density}
Under the maximal entrance law for coalescing Brownian
motions at time $T > 0$, the consecutive pattern
\[
y_0 \nwarrow x_{\nicefrac{1}{2}} \nearrow y_1
  \nwarrow \cdots
  \nwarrow x_{k-\nicefrac{1}{2}} \nearrow y_k
\]
has intensity
\[
\det\bigl(M_0(x_{\nicefrac{1}{2}}, \ldots,
x_{k-\nicefrac{1}{2}};\,
y_0, \ldots, y_k)\bigr)
\]
with respect to Lebesgue measure on~$\RR^{2k+1}$, where
$M_0$ is the $2k \times 2k$ matrix from
\Cref{prop:scaling-limit}.
\end{theorem}

\begin{proof}
At dyadic step~$n$, the starting grid has
spacing~$\gs = 2^{-n}$. By \Cref{prop:scaling-limit},
the wall-particle correlation function for $k$ walls
near $x_{\nicefrac{1}{2}}, \ldots,
x_{k-\nicefrac{1}{2}}$ and survivors near
$y_0, \ldots, y_k$ is
$\gs^k \det(M_0) + O(\gs^{k+1})$. Each factor
of~$\gs$ matches the grid spacing per wall, so
$\det(M_0)$ is the density per unit length in each
wall coordinate. Since the survivor set is locally
finite~\cite{Arratia1979} and grows monotonically with
each dyadic refinement, it almost surely stabilizes on
any bounded interval after finitely many steps.
Passing to $n \to \infty$ gives the result.
\end{proof}

\subsection{Example: reflected Brownian motion}
\label{sec:reflected-bm}
\label{ex:reflected-bm}%

The grid-refinement technique of
\Cref{sec:bm-specialization} extends naturally to any
process with continuous paths and a smooth transition
density.
As an illustration beyond standard Brownian motion,
we treat coalescing Brownian motions on the half-line
$[0, \infty)$ with reflection at~$0$. Translation invariance is lost, and the
reflecting boundary introduces a leftmost survivor
with a special role.

Reflected Brownian motion on $[0, \infty)$ has
transition density
\[
p_x(y) = \frac{1}{\sqrt{2\pi T}} \left[
  e^{-(y-x)^2/(2T)} + e^{-(y+x)^2/(2T)}
\right], \quad x, y \geq 0.
\]
This is skip-free (continuous paths), so the coalescence
determinant applies. The maximal entrance law occupies
all of $[0, \infty)$; translation invariance is broken
by the boundary at~$0$.

The set of survivors is half-infinite: there is a
leftmost survivor~$\heirpt_0$, with basin
$[0, \wallpt_{\nicefrac{1}{2}})$ bounded on the left by
the reflecting boundary. We write
\[
\heirpt_0 \nwarrow \wallpt_{\nicefrac{1}{2}}
  \nearrow \heirpt_1
  \nwarrow \wallpt_{\nicefrac{3}{2}}
  \nearrow \heirpt_2
  \nwarrow \cdots
\]
for the wall-particle system, with walls
$0 < \wallpt_{\nicefrac{1}{2}}
< \wallpt_{\nicefrac{3}{2}} < \cdots$ and survivors
$0 < \heirpt_0 < \heirpt_1 < \heirpt_2
< \cdots$.

\begin{theorem}[Half-line intensity]
\label{thm:halfline-intensity}
\label{thm:halfline-correlation}%
Under the maximal entrance law for reflected Brownian
motion on $[0, \infty)$ at time $T > 0$, the pattern
\[
y_0 \nwarrow x_{\nicefrac{1}{2}} \nearrow y_1
  \nwarrow \cdots
  \nwarrow x_{k-\nicefrac{1}{2}} \nearrow y_k
\]
with $y_0$ the leftmost survivor has intensity
\[
\det\bigl(M_0(x_{\nicefrac{1}{2}}, \ldots,
x_{k-\nicefrac{1}{2}};\,
y_0, \ldots, y_k)\bigr)
\]
with respect to Lebesgue measure
on~$(0, \infty)^{2k+1}$. Here $M_0$ is a
$(2k{+}1) \times (2k{+}1)$ matrix: its first row comes
from a particle at the boundary~$0$, and its remaining
$k$ row-pairs are the density and
$\partial_x$-derivative rows from the $k$ walls.
Columns are organized as in the coalescence pattern
$2{+}2{+}\cdots{+}2{+}1$: a $P$-and-$F$ column-pair
for each group of size~$2$, and a single $P$~column for
the final group of size~$1$.
\end{theorem}

\begin{proof}
Construct the maximal entrance law on $[0, \infty)$ by
dyadic approximation, as in \Cref{sec:maximal}: at
step~$n$, start reflected Brownian motions from every
point of $\{0, 2^{-n}, 2 \cdot 2^{-n}, \ldots\}$.

At grid spacing $\gs = 2^{-n}$, the particle at~$0$
and the $2k$ flanking particles
$a_1, b_1, \ldots, a_k, b_k$ form a system of
$2k + 1$ particles. The coalescence determinant
(\Cref{thm:coalescence-det}) gives their joint
distribution under the pattern
$2{+}2{+}\cdots{+}2{+}1$: the particle at~$0$
and~$a_1$ merge into $y_0$; each pair
$(b_l, a_{l+1})$ merges into $y_l$; and $b_k$ survives
alone at~$y_k$. Intermediate particles are trapped by
the skip-free property (as in
\Cref{thm:xy-correlation}): between~$0$ and~$a_1$,
the reflecting boundary prevents leftward escape;
between~$b_l$ and~$a_{l+1}$, the closing funnel
absorbs all intermediate particles.

The grid-refinement procedure
(\Cref{prop:scaling-limit}) carries over: each
$(a_l, b_l)$ pair collapses to a density row and a
$\partial_x$-derivative row, contributing one factor
of~$\gs$; the particle at~$0$ contributes a single
row. Thus
\[
\det(\tilde{M})
= \gs^k \det(M_0) + O(\gs^{k+1}).
\]
Since the survivor set is locally
finite~\cite{Arratia1979} and grows monotonically,
it stabilizes on any bounded interval after finitely
many steps. Passing to $n \to \infty$ gives the
intensity~$\det(M_0)$.
\end{proof}

\begin{example}[One wall on the half-line]
\label{ex:halfline-k1}
For $k = 1$ (one wall, two survivors, leftmost survivor
at~$y_0$), the pattern $2{+}1$ gives the
$3 \times 3$ matrix
\[
M_0 = \begin{pmatrix}
p_0(y_0) & F_0(y_0) - 1 & p_0(y_1) \\[0.3em]
p_{x_{\nicefrac{1}{2}}}(y_0)
  & F_{x_{\nicefrac{1}{2}}}(y_0)
  & p_{x_{\nicefrac{1}{2}}}(y_1) \\[0.3em]
\partial_x p_{x_{\nicefrac{1}{2}}}(y_0)
  & \partial_x F_{x_{\nicefrac{1}{2}}}(y_0)
  & \partial_x p_{x_{\nicefrac{1}{2}}}(y_1)
\end{pmatrix},
\]
where $p_x(y)$ and $F_x(y)$ use the reflected Brownian
motion transition density.
\end{example}

\section{Gap Distributions}
\label{sec:gap-distributions}

We compute gap distributions in the discrete setting
(where no limiting procedure is needed) and for Brownian
motion under the maximal entrance law. All results are
stated in terms of the \emph{gap intensity measure}:
$\mu(\{g\})$ is the expected number of gaps of size~$g$
per unit length. Dividing by the total intensity
$\sum_g \mu(\{g\})$ (the survivor density) recovers the
gap probability distribution.

\subsection{Single gap}\label{sec:gaps}

We derive the single-gap intensity measure, first in the discrete
setting (where the formula is exact) and then for Brownian motion
under the maximal entrance law. The Brownian result recovers the
Rayleigh law (see \Cref{sec:intro-results} for prior work); we
give a new proof through the wall-particle system.

\subsubsection{Discrete single-gap distribution}
\label{sec:discrete-single-gap}

Apply the wall-particle system with $k = 1$: a single wall
at half-integer~$x$ separates two basins, with flanking sites
$a = x - \tfrac{1}{2}$ and $b = x + \tfrac{1}{2}$ (so $b = a + 1$).
The two survivors $y_0 < y_1$ satisfy the Karlin--McGregor
non-intersection condition, and the coalescence determinant gives:
\[
\det(\tilde{M}) = P(a, y_0)\, P(b, y_1) - P(a, y_1)\, P(b, y_0).
\]
Summing the wall-particle intensity over all wall
positions gives the probability that $y_0$ and $y_1$
are consecutive survivors:
\begin{equation}\label{eq:discrete-gap-general}
\PP(y_0, y_1 \text{ consecutive survivors})
= \sum_{a \in \ZZ}
  \bigl[P(a, y_0)\, P(a{+}1, y_1)
  - P(a, y_1)\, P(a{+}1, y_0)\bigr].
\end{equation}
This formula is valid for any skip-free process with
every site initially occupied: no translation invariance
or symmetry is needed.

\medskip
\noindent\textbf{Translation-invariant case.}
Assume $P(x, y) = P(0, y - x)$ for all $x, y \in \ZZ$;
we write $P(n) = P(0, n)$ when the time~$T$ is understood.
Writing $g = y_1 - y_0$ for the gap, translation
invariance gives $P(a, y_0) = P(y_0 - a)$,
$P(a{+}1, y_1) = P(y_1 - a - 1)$, and so on; each
summand in~\eqref{eq:discrete-gap-general} depends only
on $g$ and $a - y_0$, so the sum is independent
of~$y_0$. Define the \emph{autocorrelation}
\[
R(m) = \sum_{s \in \ZZ} P_T(s)\, P_T(s + m).
\]
Then the two sums
in~\eqref{eq:discrete-gap-general} are
$R(g - 1)$ and $R(g + 1)$ respectively, giving
\begin{equation}\label{eq:discrete-gap-simplified}
\mu(\{g\}) = R(g - 1) - R(g + 1).
\end{equation}

\medskip
\noindent\textbf{Symmetric case.}
When the walk is \emph{symmetric} ($P(n) = P(-n)$ for
all~$n$), the autocorrelation reduces to a convolution:
$R(m) = \sum_s P(s)\, P(m - s) = P_{2T}(m)$ by the
Chapman--Kolmogorov identity, so
$\mu(\{g\}) = P_{2T}(g - 1) - P_{2T}(g + 1)$.
The continuous-time simple random walk ($\pm 1$ jumps,
each at rate~$1$) is the main example: every integer is
reachable from every integer, so there is no parity
constraint.

\begin{theorem}[Discrete gap intensity measure]
\label{thm:discrete-gap-body}
For a symmetric translation-invariant coalescing skip-free
process on~$\ZZ$ with every site initially occupied, the gap
intensity measure is
\[
\mu(\{g\})
= P_{2T}(g - 1) - P_{2T}(g + 1),
\qquad g = 1, 2, 3, \ldots
\]
The total intensity is a telescoping sum:
\[
\sum_{g=1}^{\infty} \mu(\{g\})
= P_{2T}(0) + P_{2T}(1),
\]
since $P_{2T}(g) \to 0$ as $g \to \infty$; this gives the
survivor density per site. Dividing by the total intensity recovers
the probability mass function
$\PP(G = g) = \mu(\{g\}) / [P_{2T}(0) + P_{2T}(1)]$.
\end{theorem}

\begin{proof}
Equation~\eqref{eq:discrete-gap-simplified} gives
$\mu(\{g\})$. For the total intensity, the partial sum
\[
\sum_{g=1}^{N}
  \bigl[P_{2T}(g - 1) - P_{2T}(g + 1)\bigr]
= P_{2T}(0) + P_{2T}(1) - P_{2T}(N) - P_{2T}(N+1)
\]
telescopes (reindex: the first sum runs over $m = 0, \ldots, N-1$
and the second over $m = 2, \ldots, N+1$). Letting
$N \to \infty$ gives total intensity $P_{2T}(0) + P_{2T}(1)$.
\end{proof}

\subsubsection{Scaling limit preview}
\label{sec:scaling-preview}

Under diffusive scaling, the discrete gap distribution from
\Cref{thm:discrete-gap-body} recovers the Brownian motion Rayleigh
density from \Cref{thm:rayleigh-intro}. The scaling sends:
\begin{itemize}
\item lattice spacing $\gs \to 0$;
\item discrete gap $g \in \ZZ$ to continuous gap
  $G = g \cdot \gs$ (measured in the rescaled coordinates);
\item discrete time $t$ to continuous time with
  $\gs^2 t$ held fixed.
\end{itemize}
In this regime, the transition probability
$P_{2t}(n)$ at the integer $n = (g \pm 1)$ is well approximated by
the Gaussian density $p(n\gs,\, 2\gs^2 t)$ times the lattice
spacing $\gs$. The difference at unit spacing becomes a
derivative:
\[
P_{2t}(g - 1) - P_{2t}(g + 1)
\approx -2\gs^2 \,\partial_G\, p(G,\, 2T)
\bigg|_{G = g\gs,\, T = \gs^2 t},
\]
and since $\partial_G\, p(G, 2T)
= -\frac{G}{2T}\, p(G, 2T)$, this gives
\[
\mu(dG) \propto G \, e^{-G^2/(4T)}\, dG,
\]
the gap intensity measure, identifying the
$\mathrm{Rayleigh}(\sqrt{2})$ family.

This argument is a scaling heuristic, not a rigorous proof: a
complete derivation requires controlling the error terms in the
Gaussian approximation and the convergence of the normalizing
constants. The rigorous Brownian motion proof follows.

\subsection{Brownian motion single gap}

We prove \Cref{thm:rayleigh-intro} by applying the $(\wallseq, \heirseq)$
framework from \Cref{sec:xy-system} (with the Brownian motion
specialization from \Cref{sec:bm-specialization}) with $k = 1$
and passing to the maximal entrance law (\Cref{sec:maximal}).

We recall the \emph{Rayleigh distribution} with scale parameter
$\sigma > 0$:
\[
\mathrm{Rayleigh}(\sigma) : \quad
f_\sigma(G) = \frac{G}{\sigma^2} e^{-G^2/(2\sigma^2)}, \quad G > 0.
\]
The mean is $\sigma \sqrt{\pi/2}$ and the variance is
$(4 - \pi)\sigma^2/2$.

\begin{theorem}[Gap intensity measure under the maximal entrance law]
\label{thm:gap-measure}
Under the maximal entrance law, the gap intensity measure in
rescaled coordinates has density
\[
\mu(dG) = \frac{G}{2\sqrt{\pi}}\, e^{-G^2/4}\, dG, \quad G > 0.
\]
Normalizing to a probability distribution gives the
$\mathrm{Rayleigh}(\sqrt{2})$ density
$f(G) = \frac{G}{2}\, e^{-G^2/4}$.
\end{theorem}

\begin{corollary}[Density of surviving particles]
\label{cor:density}
The total intensity is
$\int_0^\infty \mu(dG) = 1/\sqrt{\pi}$, giving the rescaled
survivor density. Since the rescaled mean gap is
$\EE[G] = \sqrt{\pi}$, the mean gap between consecutive
survivors at time $T$ is $\sqrt{\pi T}$, and the density of
survivors per unit length is $1/\sqrt{\pi T}$.
\end{corollary}

\begin{proof}[Proof of \Cref{thm:gap-measure}]
By \Cref{thm:bm-density} with $k = 1$, the intensity of the
pattern $y_0 \nwarrow x \nearrow y_1$ is $\det(M_0)$, where
\[
M_0 = \begin{pmatrix}
p_x(y_0) & p_x(y_1) \\
\partial_x p_x(y_0) & \partial_x p_x(y_1)
\end{pmatrix}.
\]
Using $\partial_x p_x(y) = p_x(y) \cdot (y - x)/T$, the
determinant is:
\[
\det(M_0) = p_x(y_0)\, p_x(y_1) \cdot \frac{y_1 - y_0}{T}.
\]

Change coordinates: let $u = y_0 - x$ (displacement of left
survivor from wall position) and $G = y_1 - y_0$ (gap). Then
$p_x(y_0) = (2\pi T)^{-1/2} e^{-u^2/(2T)}$ and
$p_x(y_1) = (2\pi T)^{-1/2} e^{-(u + G)^2/(2T)}$, so
\[
\det(M_0) = \frac{G}{2\pi T^2}
  \exp\!\left(-\frac{u^2 + (u+G)^2}{2T}\right).
\]
Completing the square: $u^2 + (u+G)^2 = 2(u + G/2)^2 + G^2/2$.
Integrating over $u$:
\[
\int_{-\infty}^{\infty} \det(M_0)\, du
= \frac{G}{2\pi T^2}\, e^{-G^2/(4T)} \sqrt{\pi T}
= \frac{G}{2\sqrt{\pi}\, T^{3/2}}\, e^{-G^2/(4T)}.
\]
By translation invariance, this is independent of the
location~$y_0$, so it gives the gap intensity per unit
(unrescaled) length. Changing to rescaled coordinates
$G = (y_1 - y_0)/\sqrt{T}$, the gap intensity per unit
rescaled length becomes
\[
\mu(dG) = \frac{G}{2\sqrt{\pi}}\, e^{-G^2/4}\, dG,
\]
which is $1/\sqrt{\pi}$ times the
$\mathrm{Rayleigh}(\sqrt{2})$ density.
\end{proof}

\subsection{Joint distribution of consecutive gaps}\label{sec:joint}

We derive the joint intensity of two consecutive gaps, first in
the discrete setting and then for Brownian motion, proving
\Cref{thm:joint-intro}. The Brownian
joint density was previously obtained by ben-Avraham and
Brunet~\cite{benAvrahamBrunet2005} from the IPDF hierarchy;
our derivation via the wall-particle system applies to any
skip-free process, including discrete random walks and
birth-death chains.

\subsubsection{Discrete joint gaps}

For $k = 2$ consecutive gaps, the $1{+}2{+}1$ pattern from
\Cref{ex:schema-1+2+1} gives a $4 \times 4$ determinant.
As in the single-gap case
(\Cref{sec:discrete-single-gap}), summing the wall-particle
intensity over all wall positions gives the probability
that $y_0$, $y_1$, $y_2$ are consecutive survivors:
\[
\PP(y_0, y_1, y_2 \text{ consecutive survivors})
= \sum_{a_1, a_2 \in \ZZ}
  \det(\tilde{M}),
\]
where $a_l, b_l = a_l + 1$ are the flanking sites of
wall~$l$, and $\tilde{M}$ is the $4 \times 4$ matrix
\[
\tilde{M} = \left(
\begin{array}{c|cc|c}
P(a_1, y_0) & P(a_1, y_1) & F(a_1, y_1) - 1
  & P(a_1, y_2) \\
P(b_1, y_0) & P(b_1, y_1) & F(b_1, y_1) - 1
  & P(b_1, y_2) \\
\hline
P(a_2, y_0) & P(a_2, y_1) & F(a_2, y_1)
  & P(a_2, y_2) \\
P(b_2, y_0) & P(b_2, y_1) & F(b_2, y_1)
  & P(b_2, y_2)
\end{array}
\right).
\]
By translation invariance, the sum depends only on the
gaps $g_1 = y_1 - y_0$ and $g_2 = y_2 - y_1$; in the
Brownian limit the sums over wall positions become the
integrals over $u$ and $v$ in
\Cref{thm:joint-gap} below.

\subsubsection{Brownian motion joint gaps}

We apply the
$(\wallseq, \heirseq)$ framework from \Cref{sec:xy-system}
with $k = 2$: two walls produce three survivors at
positions $y_0 < y_1 < y_2$, with consecutive gaps
$G_1 = y_1 - y_0$ and $G_2 = y_2 - y_1$.

\begin{theorem}[Joint gap intensity and negative correlation]
\label{thm:joint-gap}
The joint gap intensity for $(G_1, G_2)$ is
\[
h(G_1, G_2) = \iint_{u < v}
  \det M_0(u, v;\, G_1, G_2)\, du\, dv,
\]
where $u, v$ are the rescaled positions of the two walls
(the constraint $u < v$ is the wall ordering). The three
survivors are placed at rescaled positions $0$, $G_1$,
$G_1 + G_2$ (translation invariance fixes the leftmost
survivor at the origin). Write $S = G_1 + G_2$ for the total
gap. The $4 \times 4$ matrix~$M_0$ is:
\[
M_0 = \left(
\begin{array}{c|cc|c}
\phi(u) & \phi(G_1{-}u) & \Phi(G_1{-}u) - 1
  & \phi(S{-}u) \\[0.3em]
-u\,\phi(u) & (G_1{-}u)\,\phi(G_1{-}u) & -\phi(G_1{-}u)
  & (S{-}u)\,\phi(S{-}u) \\[0.3em]
\hline
\phi(v) & \phi(G_1{-}v) & \Phi(G_1{-}v)
  & \phi(S{-}v) \\[0.3em]
-v\,\phi(v) & (G_1{-}v)\,\phi(G_1{-}v) & -\phi(G_1{-}v)
  & (S{-}v)\,\phi(S{-}v)
\end{array}
\right),
\]
where $\phi$ is the standard normal density and $\Phi$ its
CDF. The block structure is $(2{+}2) \times (1{+}2{+}1)$
(\Cref{ex:schema-1+2+1,fig:matrix-blocks}); the even rows
are the source derivatives from the grid refinement of
\Cref{prop:scaling-limit}.

The marginal gap distributions are each
$\mathrm{Rayleigh}(\sqrt{2})$, but
the joint intensity does not factorize. The gaps are
negatively correlated, with $\rho \approx -0.163$.
\end{theorem}

\noindent\textbf{General structure for $k$ gaps.}
The joint intensity $h(G_1, \ldots, G_k)$ of $k$ consecutive
gaps is obtained by integrating a $2k \times 2k$ determinant
with the same block structure as in
\Cref{sec:finite-marginals,fig:matrix-blocks}: $k$ row-pairs
(density and derivative) and $k{+}1$ density columns plus
$k{-}1$ CDF columns.
The case $k = 2$ above has $3$ density columns and $1$ CDF
column. See \Cref{rem:higher-order} for the general case.

\begin{proof}
By \Cref{thm:bm-density}, the intensity of the consecutive
pattern $y_0 \nwarrow x_1 \nearrow y_1 \nwarrow x_2 \nearrow y_2$
is $\det(M_0)$. Fix rescaled coordinates: place the three
survivors at $0$, $G_1$, $G_1 + G_2$ (translation invariance
removes the location variable), and let $u, v$ be the rescaled
wall positions. The Gaussian density
$\phi(y - x)$ and CDF $\Phi(y - x)$ give the matrix entries
directly; the source derivative
$\partial_x \phi(y - x) = (y - x)\,\phi(y - x)$ and
$\partial_x \Phi(y - x) = -\phi(y - x)$ supply the even rows
(\Cref{rem:wronskian}). Integrating $\det(M_0)$ over the
wall positions with the ordering constraint $u < v$ yields
the joint gap intensity.
\end{proof}

The determinant does not factorize as
$f(G_1) \cdot g(G_2)$ because the wall variables $u, v$
couple the two gaps: each wall position interacts with all
three survivors. Numerical integration gives
$\rho \approx -0.163$ (to three decimal places). An
analytical proof that $\rho < 0$ remains open, as does a
closed-form expression for~$\rho$.

See \Cref{fig:joint-density}.

\begin{figure}[t]
\centering
\begin{tikzpicture}
\begin{axis}[
    view={0}{90},
    width=8cm,
    height=8cm,
    xlabel={$G_1$},
    ylabel={$G_2$},
    xmin=0, xmax=2.5,
    ymin=0, ymax=2.5,
    xtick={0, 0.5, 1.0, 1.5, 2.0, 2.5},
    ytick={0, 0.5, 1.0, 1.5, 2.0, 2.5},
    tick label style={font=\small},
    label style={font=\normalsize},
    set layers,
    grid=major,
    grid style={line width=0.2pt, draw=gray!25},
    colormap/OrRd-9,
    point meta min=-0.3,
    point meta max=0.55,
]
\addplot3[
    contour gnuplot={
        levels={0.05, 0.1, 0.15, 0.2, 0.25, 0.3,
                0.35, 0.4, 0.45, 0.5},
        labels=false,
    },
    thick,
    mesh/rows=56,
    mesh/ordering=rowwise,
] table {figures/data/joint_density_surface.dat};
\addplot3[
    contour gnuplot={
        levels={0.1, 0.2, 0.3, 0.4, 0.5},
        labels=true,
        contour label style={font=\scriptsize},
        label distance=150pt,
    },
    thick,
    mesh/rows=56,
    mesh/ordering=rowwise,
] table {figures/data/joint_density_surface.dat};
\end{axis}
\end{tikzpicture}
\caption{Joint gap intensity $h(G_1, G_2)$ from
  \Cref{thm:joint-gap}, computed by numerical integration.
  The tilted elliptical contours reflect negative correlation
  ($\rho \approx -0.163$).}
\label{fig:joint-density}
\end{figure}
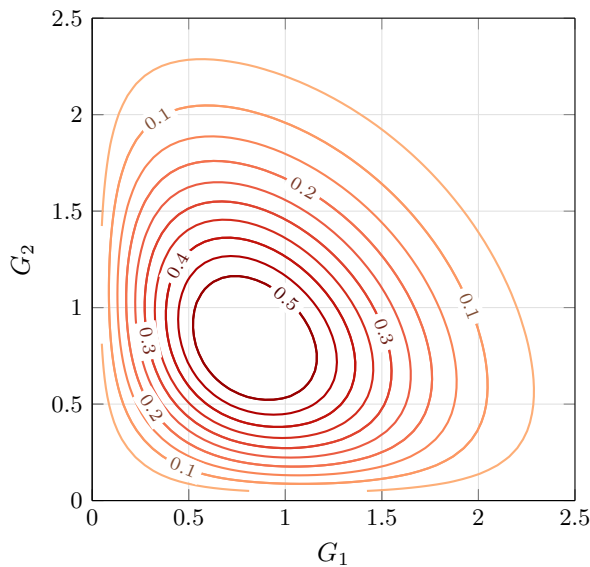

\begin{remark}[Higher-order gap distributions]
\label{rem:higher-order}
The joint intensity of $k$ consecutive gaps $(G_1, \ldots, G_k)$
follows from the $1{+}2{+}\cdots{+}2{+}1$ pattern
(\Cref{sec:xy-system}) with $2k$ particles in $k$ pairs. By
\Cref{prop:scaling-limit}, the $2k \times 2k$ determinant scales as
$\gs^k \det(M_0)$, where $M_0$ has alternating rows of transition
densities and spatial derivatives. This determines the joint
intensity $h(G_1, \ldots, G_k)$.

Preliminary numerical computations
suggest that non-adjacent gaps also have negative correlations,
though weaker than for adjacent gaps.
\end{remark}

\section{Warren's Formula}\label{sec:warren}

We now derive the determinantal CDF formula described
in \Cref{sec:intro-method} for all skip-free processes
from the coalescence determinant.

The setting is the following. Start $n$
particles at fixed positions $x_1 < \cdots < x_n$; each
performs an independent skip-free process until it meets
another particle, at which point the two coalesce and
continue as one. At time~$T$, some of the original $n$
particles have merged, so the number of distinct positions is
at most~$n$. Warren's formula gives the joint CDF of these
positions.

We use the notation $P(x, y)$ and $F(x, y)$ from
\Cref{sec:coalescence-det}. We write $\sum_z$ for
summation over the state space~$S$; in the continuous
case, all sums become integrals.

\begin{theorem}[Warren's formula for skip-free processes]
\label{thm:warren-general}
Let $X^{(1)}, \ldots, X^{(n)}$ be coalescing skip-free processes on
a linearly ordered state space $S$, starting at positions
$x_1 < x_2 < \cdots < x_n$. Let $Z_T(x_i)$ denote the position at
time $T$ of the particle that started at $x_i$
(after possible coalescence).

For $y_1 \leq y_2 \leq \cdots \leq y_n$ in $S$:
\[
\PP\bigl(Z_T(x_i) \leq y_i \text{ for all } i = 1, \ldots, n\bigr)
= \det(M^W),
\]
where the $n \times n$ matrix $M^W$ has entries:
\[
M^W_{ij} = F(x_i, y_j) - [i < j].
\]
\end{theorem}
\begin{proof}
The plan is: sum over all coalescence patterns
$n = n_1 + \cdots + n_k$, and for each pattern
sum the survivor positions $z_1, \ldots, z_k$
over the region compatible with the CDF thresholds
$y_1, \ldots, y_n$. We then show that the resulting sum
of determinants collapses to~$\det(M^W)$.

By \Cref{thm:coalescence-det}, composition
$n_1 + \cdots + n_k = n$ produces an $n \times n$
coalescence matrix~$\tilde{M}$ whose determinant gives
the joint density of the $k$ survivor positions. All
particles in block~$l$ share the same survivor~$z_l$, so
the constraint $z_l \leq y_i$ for every $i$ in the block
reduces to $z_l \leq y_{j_l}$, where
$j_l = n_1 + \cdots + n_{l-1} + 1$ is the first index
in the block (giving the smallest threshold). The CDF
decomposes as
\[
\PP\bigl(Z_T(x_i) \leq y_i \;\forall\, i\bigr)
= \sum_{n_1 + \cdots + n_k = n}\;
  \sum_{\substack{z_1 < \cdots < z_k \\[2pt]
                  z_l \leq y_{j_l}}}
  \det \tilde{M}(z_1, \ldots, z_k),
\]
a sum over all $2^{n-1}$ compositions of~$n$.

We evaluate this sum by collapsing one column at a time,
from right to left.

\medskip\noindent\textbf{Summing out the rightmost
survivor.}
We sum out the survivor variable that feeds into
column~$n$, grouping patterns in pairs. For each
composition $m_1 + \cdots + m_r = n{-}1$, define two
compositions of~$n$:
\begin{itemize}
\item Pattern~A: $m_1 + \cdots + m_r + 1$ (particle~$n$ in
  a new block);
\item Pattern~B: $m_1 + \cdots + (m_r{+}1)$ (particle~$n$
  appended to the last block).
\end{itemize}
This pairs the $2^{n-1}$ compositions of~$n$ into
$2^{n-2}$ pairs.

Columns $1, \ldots, n{-}1$ of the coalescence matrix are
identical for paired patterns~A and~B (same block sizes,
same survivors): extending the last block adds column~$n$
but does not alter the existing columns.

In Pattern~A, the survivor~$z_{r+1}$ of the new
block~$\{n\}$ appears only in column~$n$ ($P$-type).
Summing over $z_{r+1} \in (z_r,\, y_n]$ replaces this
column by the vector
\[
\bigl(F(x_i, y_n) - F(x_i, z_r)\bigr)_i.
\]
By multilinearity of the determinant in column~$n$,
Pattern~A contributes
\[
\det\bigl(\cdots,\; F_{\cdot}(y_n)\bigr)
  - \det\bigl(\cdots,\; F_{\cdot}(z_r)\bigr),
\]
where~$\cdots$ denotes columns $1, \ldots, n{-}1$
and $F_{\cdot}(y)$ is the column $(F(x_i, y))_{i=1}^n$.

In Pattern~B, column~$n$ is $F(x_i, z_r) - [i < n]$
(an $F$-column with staircase, since~$n$ is not the first
index in its block). By multilinearity:
\[
\det\bigl(\cdots,\; F_{\cdot}(z_r)\bigr)
  - \det\bigl(\cdots,\; [i < n]\bigr).
\]

Adding: the $F_{\cdot}(z_r)$ terms cancel, leaving
\[
\det\bigl(\cdots,\; F_{\cdot}(y_n) - [i < n]\bigr).
\]
Column~$n$ is now $M^W_{\cdot,n} = F(x_i, y_n) - [i < n]$,
independent of all survivor variables.

The rightmost
survivor has been summed out: column~$n$ is in its final
Warren form, and the remaining sum runs over the
$2^{n-2}$ compositions of~$n{-}1$ and their survivors.

\medskip\noindent\textbf{Iteration.}
Repeating the pairing at boundary~$(n{-}2, n{-}1)$
collapses column~$n{-}1$ to
$F(x_i, y_{n-1}) - [i < n{-}1]$, reducing to $2^{n-3}$
compositions of~$n{-}2$. After $n - 1$ iterations every
column is in Warren form. The final step sums the sole
remaining survivor~$z_1$ over $(-\infty, y_1]$, giving
column~$1 = F_{\cdot}(y_1)$ (with $[i < 1] = 0$). The
result is~$\det(M^W)$.
\end{proof}

\Cref{thm:warren-general} applies to all the skip-free
processes discussed in earlier sections, including
birth-death chains (\Cref{ex:birth-death}) and reflected
Brownian motion (\Cref{ex:reflected-bm}).

\section*{Acknowledgments}

We thank Theodoros Assiotis, Bal\'azs B\'ar\'any,
Maciej Dołęga, Sho Matsumoto, B\'alint T\'oth,
\'Akos Urb\'an, Oleg Zaboronski, and
Karol Życzkowski for stimulating discussions and
helpful literature suggestions.

We thank Richard Arratia for generously providing
access to his PhD thesis~\cite{Arratia1979}.

We are grateful to Folkmar Bornemann for sharing
his MATLAB toolbox for computing distributions in
random matrix theory~\cite{Bornemann2010}. This
software was instrumental in our preliminary
numerical experiments.

Claude Code (Anthropic) was used as an assistant during formula
discovery and manuscript preparation.

\printbibliography

\end{document}